\documentclass[12pt]{amsart}
\usepackage{accsupp}
\usepackage{amssymb,amscd,amsmath}
\usepackage{amsfonts}
\usepackage{fullpage}
\usepackage{array}
\usepackage{amsthm}
\usepackage{url}
\usepackage{mathrsfs}
\usepackage{todonotes}
\usepackage{relsize}
\usepackage{enumerate}
\usepackage{comment}
\numberwithin{equation}{section}
\newtheorem{theorem}{Theorem}[section]
\newtheorem{proposition}[theorem]{Proposition}
\newtheorem{lemma}[theorem]{Lemma}
\newtheorem{corollary}[theorem]{Corollary}

\theoremstyle{definition}
\newtheorem{definition}[theorem]{Definition}
\newtheorem{remark}[theorem]{Remark}

\newtheorem*{remark*}{Remark}
\newtheorem*{proposition*}{Proposition}
\newtheorem*{acknowledgement*}{Acknowledgement}

\newcommand{\Z}{\mathbb{Z}}
\newcommand{\Q}{\mathbb{Q}}
\newcommand{\R}{\mathbb{R}}

\DeclareMathOperator{\GL}{GL}
\DeclareMathOperator{\modulo}{mod}
\newcommand{\SL}{SL}
\newcommand{\HH}{\mathcal{H}}

\newcommand{\MF}{Maass~form~}

\newcommand{\abcd}[4]{\begin{pmatrix}#1&#2\\#3& #4\end{pmatrix}}
\newcommand{\sabcd}[4]{\left(\begin{smallmatrix}#1&#2\\#3& #4\end{smallmatrix}\right)}
\newcommand{\hyp}[4]{~_2F_1\left(\left.\begin{smallmatrix}#1,~#2\\#3\end{smallmatrix}\right\rvert#4\right)}
\renewcommand{\mod}{\text{ mod }}
\renewcommand{\Im}{\text{Im\,}}
\renewcommand{\Re}{\text{Re\,}}
\DeclareMathOperator{\Res}{Res}
\DeclareMathOperator{\sgn}{sgn}
\DeclareMathOperator{\Sym}{Sym}

\input xy
\xyoption{all}
\title{Weil's converse theorem for Maass forms and cancellation of zeros}
\author[M. Neururer]{Michael Neururer}\address{Fachbereich Mathematik, Technische Universit\"{a}t Darmstadt, Schlo\ss gartenstr. 7, 64289 Darmstadt, Germany.}
\email{neururer@mathematik.tu-darmstadt.de}
\author[T. Oliver]{Thomas Oliver}\address{Mathematical Institute, Andrew Wiles Building, University of Oxford, Radcliffe Observatory Quater, Woodstock Road, Oxford, OX2 6GG, UK and the Heilbronn Institute for Mathematical Research, Bristol, UK.}
\email{Thomas.Oliver@maths.ox.ac.uk}
\date{August 08, 2019}
\usepackage{hyperref}
\hypersetup{pdfauthor={Michael Neururer and Thomas Oliver},%
	            pdftitle={Converse Theorem},%
	            pdfsubject={Number Theory},%
	            pdfcreator={XeLaTeX},%
	            colorlinks=true,%
	            linkcolor=[rgb]{0,.6,1},
	            citecolor=blue}
\begin{document}
\maketitle
\subsection*{Abstract}
We prove two principal results. Firstly, we characterise Maass forms in terms of functional equations for Dirichlet series twisted by primitive characters. The key point is that the twists are allowed to be meromorphic. This weakened analytic assumption applies in the context of our second theorem, which shows that the quotient of the symmetric square $L$-function of a Maass newform and the Riemann zeta function has infinitely many poles.
\section{Introduction}
Given two distinct $L$-functions $L_1$ (resp. $L_2$) in the Selberg class (cf. \cite{SelbergSurvey}), let $S_1$ resp. $S_2$ denote the set of zeros in the critical strip. A lower bound for the number of elements in the symmetric difference $\left(S_1\backslash S_2\right)\bigcup\left(S_2\backslash S_1\right)$ up to a finite height was given by Murty--Murty \cite{MM}. Denote by $d_1$ (resp. $d_2$) is the degree of $L_1$ (resp. $L_2$), which is roughly speaking the number of gamma factors in the functional equation. Under certain orthogonality hypotheses, a lower bound for the asymmetric difference $S_1\backslash S_2$ was established by Bombieri--Perelli when $d_2=d_1$ \cite{BombieriPerelli}. When $d_2-d_1\leq0$, a lower bound for $S_1\backslash S_2$ was proved by Srinivas \cite{Srinivas}. Denoting the completed $L$-functions by $\Lambda_1$ (resp. $\Lambda_2$) we see in particular that the quotient $\Lambda_2(s)/\Lambda_1(s)$ has infinitely many poles when $d_2-d_1\leq0$. Much less is known in the cases where $d_2-d_1>0$.

In \cite[Corollary~1.9]{LFAD}, it was shown that if $\pi_{1}$ (resp. $\pi_{2}$) is a unitary cuspidal automorphic representation of GL$_{d_1}(\mathbb{A}_{\mathbb{Q}})$ (resp. GL$_{d_2}(\mathbb{A}_{\mathbb{Q}})$), $d_2-d_1\leq1$ and $\pi_1\ncong\pi_2$, then the quotient $\Lambda(s,\pi_2)/\Lambda(s,\pi_1)$ has infinitely many poles. Prior to Booker's results, it was established by Raghunathan that if $g$ is a non-CM holomorphic modular form with nebentypus $\chi$, then $\Lambda(\Sym^2g,s)/\Lambda(\chi,s)$ has infinitely many poles \cite{ACOZOLF}. In that case, one has $d_1=1$ and $d_2=3$, so that $d_2-d_1=2$. It is necessary to exclude the CM case. Indeed, if an elliptic curve $E$ has CM, then $L(\mathrm{Sym}^2E,s)$ is in fact divisible by the Riemann zeta function. In this paper we replace $g$ by a real-analytic Maass form $f$, under a primitivity assumption on $L(\mathrm{Sym}^2f,s)$. 

The proof works by showing that were the quotient to have finitely many poles, then it would be the $L$-function of a Maass form. Consequently, a significant portion of this paper is concerned with characterising Maass forms in terms of analytic properties of their $L$-functions and their twists, i.e., converse theorems for Maass forms. Hecke \cite{Hecke1936} established a converse theorem for modular forms of level 1 using a single functional equation, and the analogous statement for Maass forms of level 1 was proved by Maass \cite{Maass1949}. 
A converse theorem for modular forms of level $N\geq1$ was given by Weil \cite{UDBDRDF}. Weil's key idea was to assume, along with the appropriate functional equation of the $L$-function, functional equations for twists of the $L$-function. Shortly after Weil, Jacquet--Langlands proved a converse theorem for automorphic representations of $\GL_2(\mathbb{A}_{F})$ over global fields $F$ \cite{AFOGL2}. In particular, the Jacquet--Langlands converse theorem applies to holomorphic modular and Maass forms of level $N$. 

Over the years there have been various attempts to reduce the number and ramification of twists in converse theorems \cite{PSconverse}, \cite{modFormsNDirichletSeries}, \cite{Wli}, \cite{EOHCT}, \cite{DPZ}. In our proposed application, there is no shortage of twisted functional equations, and in our converse theorem it suffices to use the same twisting moduli as Weil (see also \cite[Theorem~4.3.15]{Miyake}). Our focus is rather on weakening the other analytic assumptions which are not applicable in our context and not covered by the converse theorems of Weil or Jacquet--Langlands. Similarly to the work of Booker--Krishnamurthy \cite{ASOTGL2CT, WCTWP}, we must allow the non-trivial character twists to have arbitrary poles and assume that the trivial twist has at most finitely many poles. We must also avoid Euler products, as in general the Euler factors of quotients will not be of standard form. Altogether, our analytic assumptions are so weak that they can be applied to quotients of automorphic $L$-functions for a contradiction. We note that Raghunathan proved a version of the Hecke--Maass converse theorem for level 1 Maass forms allowing for polar $L$-functions in \cite{OLFWPSMFE}. Our theorem is different in that it assumes twisted functional equations rather than an Euler product, and applies to arbitrary level $N$. 

The proof of our converse theorem works by showing that meromorphy and the twisted functional equations imply that the twists are holomorphic away from a small set of special points. To see this, we use the asymptotics of hypergeometric functions to study the Taylor expansions of Fourier--Whittaker series. Once we have reduced to the entire case, our argument closely follows Weil's with one crucial difference. Let $\mathcal{H}$ denote the upper half-plane. Weil noticed that if a holomorphic function $F:\mathcal{H}\rightarrow\mathbb{C}$ is invariant under an infinite order elliptic matrix\footnote{Such a matrix can not have integral coefficients, as the elliptic matrices in $\SL_2(\mathbb{Z})$ are of finite order.} in SL$_2(\mathbb{R})$ acting via the weight $k$ slash operator for non-zero $k$, then it vanishes identically on $\mathcal{H}$. Weil applied this to functions of the form $F=f-f|\gamma$, where $f$ is a holomorphic Fourier series and $\gamma\in\Gamma_0(N)$. In the real analytic setting, there exist non-constant functions that are invariant under infinite order elliptic matrices. The new insight in our proof, the method of ``two circles'', was suggested to the authors by David Farmer. If a real analytic, or even a continuous, function is invariant under two infinite order elliptic matrices, then it is constant (see Theorem \ref{TwoCircles.theorem}). Following Weil's approach we construct two such matrices under which $F$ is invariant from a sufficiently large set of twisted functional equations. 

We briefly review some standard notation. Recall the archimedean Euler factor of the Riemann zeta function:
\begin{equation}\label{eq:gammaR}
\Gamma_\R(s) = \pi^{-s/2}\Gamma(s/2).
\end{equation}
Given a Dirichlet character $\psi$ mod $q$, one has the associated Gauss sum:
\begin{equation}\label{eq:GaussSum}
\tau(\psi)=\sum_{a\mod q}\psi(a)e^{2\pi i\frac{a}{q}}.
\end{equation}
We will encounter the Whittaker function, 
\begin{equation}
W_\nu(u)=4\sqrt{|u|}K_v(2\pi |u|),
\end{equation}
where  is the Bessel function given explicitly by the following integral:
\begin{equation}
K_v(u)=\frac12\int_0^\infty e^{-|u|(t+t^{-1})/2}t^v\frac{dt}{t}.
\end{equation}
We conclude this introduction by stating our principal theorems. The following is proved in Section~\ref{PolarConverse.subsection}.
\begin{theorem}\label{thm:PolarConverse}
Let $N$ be a positive integer, $\chi$ be a Dirichlet character mod $N$, $\epsilon\in\{0,1\}$, $\nu\in\mathbb{C}\backslash\{0\}$ be such that $\frac{1}{4}-\nu^2>0$, and $a_n,b_n$ be sequences of complex numbers such that $|a_n|,|b_n|=O\left(n^{\frac{1}{2}+\kappa}\right)$ for some $0<\kappa<\frac12-|\Re\nu|$. For all $q$ relatively prime to $N$, primitive Dirichlet characters $\psi$ modulo $q$, and $k$ such that $\psi(-1)=(-1)^k$, define
\begin{equation}\label{eq:L}
L_f(s,\psi)=\sum_{n=1}^{\infty}\psi(n)a_nn^{-s},~L_g(s,\bar{\psi})=\sum_{n=1}^{\infty}\overline{\psi(n)}b_nn^{-s},
\end{equation}
and, 
\begin{align}\label{eq:Lambda}\begin{split}
\Lambda_f(s,\psi)&=\Gamma_\R(s+[\epsilon+k]+\nu)\Gamma_\R(s+[\epsilon+k]-\nu)L_f(s,\psi),\\
\Lambda_g(s,\bar{\psi})&= \Gamma_\R(s+[\epsilon+k]+\nu)\Gamma_\R(s+[\epsilon+k]-\nu)L_g(s,\bar{\psi}),
\end{split}\end{align}
where $[\epsilon+k]\in\{0,1\}$ is chosen to be equal to $\epsilon+k$ modulo $2$.  If $\psi=\textbf{1}$ is the trivial character we omit it from the notation. Let $\mathcal{P}$ be a set of odd primes coprime to $N$ such that the congruence $p\equiv u\bmod v$ has infinitely many solutions $p\in\mathcal{P}$ for all $u\in\Z$ and $v\in\Z_{>0}$ with $(u,v)=1$. For all primitive Dirichlet characters of modulus $q\in\{1\}\cup\mathcal{P}$ assume $\Lambda_f(s,\psi)$ and $\Lambda_g(s,\psi)$ continue to meromorphic functions on $\mathbb{C}$ and satisfy the functional equation
\begin{equation}\label{eq:FE}
\Lambda_f(s,\psi)=(-1)^{\epsilon}\psi(N)\chi(q)\frac{\tau(\psi)}{\tau(\overline{\psi})}(q^2N)^{\frac12-s}\Lambda_g(1-s,\bar{\psi}).
\end{equation}
If there is a non-zero polynomial $P(s)\in\mathbb{C}[s]$ such that $P(s)\Lambda_f(s)$ continues to an entire function of finite order, then $\Lambda_f(s)$ and $\Lambda_g(s)$ are analytic on $\mathbb{C}-\{\pm\nu,1\pm\nu\}$, with at most simple poles in the set $\{\pm\nu,1\pm\nu\}$, and the following series define weight 0 Maass forms on $\Gamma_0(N)$ of parity $\epsilon$, nebentypus $\chi$ (resp. $\overline{\chi}$) and eigenvalue $\frac{1}{4}-\nu^2$:
\begin{equation}\label{eq:FWseries}
f(z)=f_0(z) + \tilde{f}(z),~g(z)=g_0(z) + \tilde{g}(z),
\end{equation}
where 
\begin{equation}\label{eq:tilde}
\tilde{f}(z):=\sum_{n\neq 0}\frac{a_n}{2\sqrt{|n|}}W_{\nu}(ny)e(nx),~\tilde{g}(z):=\sum_{n\neq 0}\frac{b_n}{2\sqrt{|n|}}W_{\nu}(ny)e(nx),
\end{equation}
in which, for $n\geq0$, $a_{-n}=(-1)^{\epsilon}a_n$, $b_{-n}=(-1)^{\epsilon}b_n$, and
\begin{align}\label{eq:constneq0}
f_0(z)&=- \Res_{s=-\nu}\Lambda_f(s) y^{\frac12+\nu} -\Res_{s=\nu}\Lambda_f(s) y^{\frac12-\nu},
\\
g_0(z)&=N^{\frac12+\nu}\Res_{s=1+\nu}\Lambda_g(s) y^{\frac12+\nu} + N^{\frac12-\nu}\Res_{s=1-\nu}\Lambda_g(s) y^{\frac12-\nu}.
\end{align}
Furthermore $f(z) = g(-1/Nz)$ for all $z\in\HH$.
\end{theorem}
The assumption that $\sigma=\frac12+\kappa$ is first used in Lemma~\ref{Residues.lemma}. This assumption is sufficient for our main application, but should be removed with further work. The assumption in Theorem~\ref{thm:PolarConverse} that $\nu\neq0$ will be removed in a follow-up paper. The following is proved in Section~\ref{sec.Quotients}.

\begin{corollary}\label{thm:Quotients}
Let $f$ be a weight 0 Maass newform on $\Gamma_0(N)$ such that $\nu\neq0$ and $\mathrm{Sym}^2f$ is cuspidal, and let $\xi(s)=\pi^{-s/2}\Gamma(s/2)\zeta(s)$ denote the completed Riemann zeta function. The quotient $\Lambda\left(\Sym^2f,s\right)/\xi(s)$ has infinitely many poles.
\end{corollary}
We note that criteria for the cuspidality of $\mathrm{Sym}^2f$ are given in \cite[Theorem~9.3]{RBAR23}.

\subsection*{Acknowledgement} We are grateful to Andrew Booker for his comments throughout this project. We thank David Farmer, whose ``two circles'' idea is an important ingredient in the proof of Theorem~\ref{thm:EntireConverse}, and Giuseppe Molteni, whose careful reading of an earlier version lead to several improvements.

\section{Preliminaries}
\subsection{The Mellin transform}\label{Mellin_transform.subsection}
The Mellin transform of a function $\phi$ on the positive real axis is given by $\mathcal{M}(\phi)(s)=\int_0^\infty \phi(t) t^{s-1}dt$. We will often use the Mellin transform shifted by $1/2$:
\[
\widetilde{\mathcal{M}}(\phi)(s) =\mathcal{M}(\phi)\left(s-\frac12\right)= \int_0^\infty \phi(t)t^{s-\frac12}\frac{dt}{t},
\]
Suppose $\phi$ is of rapid decay at $\infty$ and grows like $t^{-A}$ for some real number $A$ as $t\to 0$. Then $\widetilde{\mathcal{M}}(\phi)$ is holomorphic in the half plane $\Re s > A+\frac12$. The inverse of $\widetilde{\mathcal{M}}$ is given by
\[
\widetilde{\mathcal{M}}^{-1}(f)(t) = \frac{1}{2\pi i}\int_{\left(\sigma-\frac12\right)}f(s)t^{\frac12-s}ds,
\]
where $\sigma\gg 0$ and $(\sigma-\frac12)$ is the path $\sigma-\frac12 + it$ with $t$ going from $-\infty$ to $\infty$.
\subsection{Maass forms}\label{Maass_forms.subsection}
In this section we recall some basic facts about Maass forms. For a function $f:\HH\to\mathbb{C}$ and $\gamma=\sabcd abcd\in\GL_2^+(\R)$ we write $f|\gamma(z) = f\left(\frac{az+b}{cz+d}\right)$. A \MF of weight zero on $\Gamma_0(N)$ is an eigenfunction on $\Gamma_0(N)\backslash\mathcal{H}$ of the Laplace-Beltrami operator with eigenvalue\footnote{The Selberg eigenvalue conjecture asserts that $\nu$ is purely imaginary.} $\frac14-\nu^2>0$, satisfying certain growth conditions and the transformation rule
\[
f|\gamma=\chi(d)f,~\forall\gamma=\begin{pmatrix}a&b\\c&d\end{pmatrix}\in\Gamma_0(N),
\]
for a Dirichlet character $\chi$ mod $N$. A (weight $0$) \MF $f$ on $\Gamma_0(N)$ has a Fourier expansion of the form
\begin{align}\label{FourierExpansionW.equation}
f(z)=f_0(y)+\sum_{n\neq 0} \frac{a_n}{2\sqrt{|n|}} W_\nu(n y)e(nx),
\end{align}
where $a_n\in\mathbb{C}$ ($n\neq0$), $W_\nu(u)=4\sqrt{|u|}K_v(2\pi |u|)$ is the Whittaker function and $K_v(u)=\frac12\int_0^\infty e^{-|u|(t+t^{-1})/2}t^v\frac{dt}{t}$ is a Bessel function. The term $f_0(y)$ will be made explicit below.  By diagonalising with respect to the involution $\iota:z\mapsto -\overline{z}$ on $\mathcal{H}$, we may assume that $f$ is either even ($a_n = a_{-n}$) or odd ($a_n = -a_{-n}$). The \emph{parity} $\epsilon\in\{0,1\}$ of $f$ is $0$ if $f$ is even and $1$ if it is odd. By $\cos^{(k)}$ we denote the $k$-th derivative of $\cos$. If $f$ has parity $\epsilon,$ then 
\begin{equation}\label{FourierExpansionK.equation}
f(z)=f_0(y) + (-i)^\epsilon\sum_{n=1}^{\infty} \frac{a_n}{\sqrt{n}} W_\nu(n y)\cos^{(\epsilon)}(2\pi n x),
\end{equation}
where
\begin{align}\label{eqn:constant term}
f_0(y) = \begin{cases} a_0 y^{\frac12+\nu} + a_0' y^{\frac12-\nu},&\epsilon=0,\nu\neq 0,\\
a_0 y^{\frac12} + a_0' y^{\frac12}\log y,&\epsilon=0,\nu= 0,\\
0,& \epsilon=1,
\end{cases}
\end{align}
for $a_0,a_0'\in\mathbb{C}$. To a \MF we can associate an $L$-function
a priori defined on the right half plane $\Re s>3/2$ by the Dirichlet series
\begin{equation}
L_f(s)=\sum_{n=1}^\infty a_n n^{-s}.
\end{equation}
The completed $L$-function of $f$ is defined as
\[
\Lambda_f(s) =\Gamma_\R(s+\epsilon+\nu)\Gamma_\R(s+\epsilon-\nu)L_f(s).
\]
The following facts are standard: If $\epsilon=1$, then the function $\Lambda_f(s)$ continues to an entire function; if $\epsilon=0$, then $\Lambda_f(s)$ has meromorphic continuation to $\mathbb{C}$ with possible simple poles in the set $\{\pm\nu,1\pm\nu\}$ if $\nu\neq 0$ and at most double poles in $\{0,1\}$ if $\nu=0$. Setting $g(z)=f(-1/Nz)$ the $L$-functions $\Lambda_f(s)$ and $\Lambda_g(s)$, along with their twists, satisfy the functional equations \eqref{eq:FE}. 
\section{Weil's converse theorem for Maass forms}
We need the following as a stepping stone to Theorem~\ref{thm:PolarConverse}. Unlike Theorem~\ref{thm:PolarConverse}, it applies when $\nu=0$ and assumes only polynomial growth for the sequences $\{a_n\},\{b_n\}$.
\begin{theorem}\label{thm:EntireConverse}
Let $N$, $\chi$, $\epsilon$ and $\mathcal{P}$ be as in Theorem~\ref{thm:PolarConverse}, let $\nu\in\mathbb{C}$ be such that $\frac{1}{4}-\nu^2>0$, and let $a_n,b_n$ be sequences of complex numbers such that $|a_n|,|b_n|=O(n^{\sigma})$ for some $\sigma\in\R$. Assume that:
\begin{enumerate}
\item For primitive characters $\psi$ of conductor $q\in\mathcal{P}$ the functions, $\Lambda_f(s,\psi)$ and $\Lambda_g(s,\psi)$ continue to entire functions,
\item If $\epsilon=1$ then $\Lambda_f(s)$ and $\Lambda_g(s)$ continue to entire functions,
\item If $\epsilon = 0$ and $\nu\neq 0$ then $\Lambda_f(s)$ and $\Lambda_g(s)$ continue to meromorphic functions on $\mathbb{C}$ with at most simple poles in the set $\{1\pm\nu,\pm\nu\}$,
\item  If $\epsilon = 0$ and $\nu= 0$ then $\Lambda_f(s)$ and $\Lambda_g(s)$ continue to meromorphic functions on $\mathbb{C}$ with at most double poles in the set $\{0,1\}$,
\end{enumerate}
and, for all primitive characters $\psi$ of conductor $q\in\mathcal{P}\cup\{1\}$, the functions $\Lambda_f(s,\psi),\Lambda_g(s,\psi)$ are uniformly bounded on every vertical strip outside of a small neighbourhood around each pole and satisfy the functional equation~\eqref{eq:FE}. For $n\geq0$, define $a_{-n}=(-1)^{\epsilon}a_n$, $b_{-n}=(-1)^{\epsilon}b_n$.  If $\nu\neq0$, define $f_0(z),g_0(z)$ as in equation~\eqref{eq:constneq0}. If $\nu=0$, define
\begin{align}\label{eq:consteq0}
f_0(z)&= -\Res_{s=0}\Lambda_f(s) y^{\frac12} + \Res_{s=0}s\Lambda_f(s) y^{\frac12}\log y,\\
g_0(z)&=-\Res_{s=0}\Lambda_g(s) y^{\frac12} + \Res_{s=0}s\Lambda_g(s) y^{\frac12}\log y.
\end{align}
If $f(z),g(z)$ are defined as in equation~\eqref{eq:FWseries}, then $f(z),g(z)$ define weight 0 Maass forms on $\Gamma_0(N)$ of parity $\epsilon$, nebentypus $\chi$ (resp. $\overline{\chi}$) and eigenvalue $\frac{1}{4}-\nu^2$. Furthermore $f(z) = g(-1/Nz)$ for all $z\in\HH$.
\end{theorem}
The above differs from the converse theorem in \cite{MSSU}, as it requires only primitive twists. 
From now on assume that $a_n$, $b_n$ are sequences as in Theorem \ref{thm:EntireConverse}. We define $f$, $g$ and the twisted $L$-functions associated to $f$ and $g$ as in Theorem~\ref{thm:PolarConverse}.
\subsection{Additive Twists}
Let $q\in\mathcal{P}$ be a prime such that $(q,N)=1$ and set $\alpha=\frac{a}{q}\in\mathbb{Q}$ for some $a\in\mathbb{Z}$. We use the following notation as in \cite{BCK}. 
\begin{definition}\label{AddTwists.definition}
Let $k\in\mathbb{Z}_{\geq0}$. For $\alpha\in\mathbb{Q}^{\times}$, the \textsl{additive twists} of $L_f(s)$ by $\alpha$ are
\begin{align}\label{eq:CosTwist}
L_{f}\left(s,\alpha,\cos^{(k)}\right) =& \sum_{n=1}^{\infty} \cos^{(k)}(2\pi n\alpha) a_n n^{-s}.
\end{align}
\end{definition}
Up to sign, Definition~\ref{AddTwists.definition} depends only on $k$ modulo 2.
\begin{lemma}\label{lem:Lcos and Lsin via twists}
We have
\begin{align*}
L_f\left(s,\alpha,\cos^{(k)}\right)=&\frac{i^k}{q-1}\sum_{\substack{\psi~(\modulo q)\\\psi\neq\psi_0\\\psi(-1)=(-1)^k}}\tau(\overline{\psi})\psi(a) L_f(s,\psi)\\
&+\begin{cases}
(-1)^{k/2}[L_f(s)-\frac{q}{q-1}L_f(s,\psi_0)], & k\text{ even,}\\
0, & k\text{ odd},
\end{cases}
\end{align*}
where the sums are over Dirichlet characters modulo $q$, and $\psi_0$ denotes the trivial Dirichlet character mod $q$.
\end{lemma}
\begin{proof}
This follows from
\begin{align*}
\cos\left(\frac{2\pi n a}{q}\right) &= 1 - \frac{q}{q-1}\psi_0(n) + \frac{1}{q-1}\sum_{\substack{\psi~(\text{mod q})\\\psi\neq\psi_0\\\psi(-1)=1}}\tau(\overline{\psi})\psi(an),\\
\sin\left(\frac{2\pi n a}{q}\right) &= -\frac{i}{q-1}\sum_{\substack{\psi~(\text{mod q})\\\psi(-1)=-1}}\tau(\overline{\psi})\psi(an) .
\end{align*}
\end{proof}
Define
\begin{equation}\label{eq:gamma}
\gamma_f^{(-)^k}(s)=\Gamma_\R\left(s+[k+\epsilon]+\nu\right)\Gamma_\R\left(s+[k+\epsilon]-\nu\right),
\end{equation}
where $(-)^k$ denotes $+$ if $k$ is even and $-$ if $k$ is odd. As explained in the introduction, for $m\in\Z$ we write $[m]$ for the element in $\{0,1\}$ with the same parity as $m$. We see that $\Lambda_f(s)=\gamma_f^+(s)L_f(s)$.
We define the completion of the additive twists by
\begin{align}\label{eq:addcompletions}
\begin{split}
&\Lambda_f\left(s,\alpha,\cos^{(k)}\right)=
\gamma_f^{(-)^k}(s)L_f\left(s,\alpha,\cos^{(k)}\right)\\
&=
\frac{i^k}{q-1}\sum_{\substack{\psi~(\modulo q)\\\psi\neq\psi_0\\\psi(-1)=(-1)^k}}\tau(\overline{\psi})\psi(a)\Lambda_f(s,\psi)
+\begin{cases}
(-1)^{k/2}[\Lambda_f(s)-\frac{q}{q-1}\Lambda_f(s,\psi_0)], & k\text{ even,}\\
0, & k\text{ odd}.
\end{cases}
\end{split}
\end{align}
\begin{proposition}\label{prop: SinCosFE} 
The additive twists satisfy the following functional equations
\begin{align*}
\Lambda_f\left(s,\alpha,\cos^{(k)}\right)=&(-1)^\epsilon\frac{i^k(q^2N)^{\frac12-s}\chi(q)}{q-1}\sum_{\substack{\psi~(\text{mod }q)\\\psi\neq\psi_0\\\psi(-1)=(-1)^k}}\psi(Na)\tau(\psi)\Lambda_g(1-s,\overline{\psi})\\
&+\begin{cases}
(-1)^{k/2}[\Lambda_f(s)-\frac{q}{q-1}\Lambda_f(s,\psi_0)], & k\text{ even,}\\
0, & k\text{ odd.}
\end{cases}
\end{align*}
\end{proposition}
For the following Lemma, recall the (Gauss) hypergeometric function $_2F_1$ which is reviewed in Appendix~\ref{Hypergeometric.subsection}.
\begin{lemma}[6.699(3-4) in \cite{IntegralTables}]\label{Lem:Mellin}
Let $w\in\mathbb{R}$, $k\geq0$, and $\nu\in\mathbb{C}$. One has, for $\Re(-s\pm\nu)<\epsilon$,
\[
4\int_0^{\infty}K_{\nu}(2y)\cos^{(k)}(2wy)y^s\frac{dy}{y}
=i^k(2w)^{[k]}\pi^{s}\gamma_f^{(-)^{k+\epsilon}}(s)\hyp{\frac{s+\nu+[k]}{2}}{\frac{s-\nu+[k]}{2}}{\frac12+[k]}{-w^2}.
\]
\end{lemma} 
\begin{proposition}\label{prop:Mellin-transform}
Let $h(z)=(-i)^\epsilon\sum_{n= 1}^\infty \frac{c_n}{\sqrt{n}} W_{\nu}(ny)\cos^{(\epsilon)}(2\pi nx)$ with polynomially bounded $c_n$, $\epsilon\in\{0,1\}$, $w\in\R$ and $\alpha\in\Q$. Define $\Lambda_h(s)$ and its twists just as in \eqref{eq:Lambda} and \eqref{AddTwists.definition} with $a_n$ replaced by $c_n$. One has
\begin{multline}\label{eq:Mellin-transform}
\int_0^{\infty}h(iy+wy+\alpha)y^{s-\frac{1}{2}}\frac{dy}{y}\\
=\sum_{j\in\{0,1\}}i^{-j}(2w)^{[j+\epsilon]}\Lambda_h\left(s,\alpha,\cos^{(j)}\right)\hyp{\frac{s+\nu+[j+\epsilon]}{2}}{\frac{s-\nu+[j+\epsilon]}{2}}{\frac12+[j+\epsilon]}{-w^2}.
\end{multline}
\end{proposition}
\begin{proof}
The result follows readily from the following computation which uses Lemma~\ref{Lem:Mellin}:
\begin{align*}
&\int_0^{\infty}h(iy+wy+\alpha)y^{s-\frac{1}{2}}\frac{dy}{y}
\\
&=4(-i)^{\epsilon}\sum_{n=1}^{\infty} c_n \int_0^{\infty} K_{\nu}(2\pi n y)\cos^{(\epsilon)}(2\pi n(wy+\alpha))y^{s}\frac{dy}{y}\\
&=4(-i)^{\epsilon}\sum_{n=1}^{\infty} c_n \sum_{j\in\{0,1\}}(-1)^{j}\cos^{(j)}(2\pi n \alpha)\int_0^{\infty} K_{\nu}(2\pi n y)\cos^{(\epsilon+j)}(2\pi n wy)y^{s}\frac{dy}{y}\\
&=4(-i)^{\epsilon}\pi^{-s}\sum_{n=1}^{\infty} c_n n^{-s} \sum_{j\in\{0,1\}}(-1)^{j}\cos^{(j)}(2\pi n \alpha)\int_0^{\infty} K_{\nu}(2 y)\cos^{(\epsilon+j)}(2 wy)y^{s}\frac{dy}{y}\\
&=\sum_{n=1}^{\infty} c_n n^{-s} \sum_{j\in\{0,1\}}i^{j}\cos^{(j)}(2\pi n \alpha)
(2w)^{[j+\epsilon]}\gamma_f^{(-)^j}(s)\hyp{\frac{s+\nu+[j+\epsilon]}{2}}{\frac{s-\nu+[j+\epsilon]}{2}}{\frac12+[j+\epsilon]}{-w^2}.
\end{align*}
\end{proof}
\subsection{Transformation properties from the functional equation of $\Lambda_f$}\label{subsection:transformation props from FE}
We first show that $f(z)=g(-1/Nz)$ follows from the functional equation of $\Lambda_f(s)$. Define $\tilde{f},\tilde{g}$ by equation \eqref{eq:tilde}. Let $z=wy+iy\in\mathcal{H}$ with $w\in\mathbb{R}_{\geq0}$. If $c>1+|\Re\nu|$, then by Proposition~\ref{prop:Mellin-transform} for $\alpha=0$ we can obtain $\tilde{f}(wy+iy)$ as an inverse Mellin transform.
\begin{equation}\label{eq:relating f and g(-1/Nz)}
\tilde{f}(wy+iy)=\frac{(2 w)^{\epsilon}}{2\pi i}\int_{(c)}\Lambda_f(s)\hyp{\frac{s+\epsilon+\nu}{2}}{\frac{s+\epsilon-\nu}{2}}{\frac{1}{2}+\epsilon}{-w^2}y^{\frac12-s}ds.
\end{equation}
Shifting the path of integration to the left this equals
\[
\frac{(2w)^{\epsilon}}{2\pi i}\int_{(1-c)}\Lambda_f(s)\hyp{\frac{s+\epsilon+\nu}{2}}{\frac{s+\epsilon-\nu}{2}}{\frac{1}{2}+\epsilon}{-w^2}y^{\frac12-s}ds+H(z),
\]
where
\begin{align*}
H(z) =
\begin{cases}
0,&\epsilon=1,\\
\sum_{x\in\{1\pm\nu,\pm\nu\}}\Res_{s=x}\hyp{\frac{s+\nu}{2}}{\frac{s-\nu}{2}}{\frac12}{-w^2}\Lambda_f(s)y^{\frac12-s}, & \epsilon=0.
\end{cases}
\end{align*}
Now we apply the functional equation to $\Lambda_f(s)$ and the Euler identity \eqref{eqn:euler identity} to obtain
\begin{align*}
\tilde{f}(wy+iy)=&\frac{(2w)^{\epsilon}}{2\pi i}\int_{(c)}\Lambda_f(1-s)\hyp{\frac{1-s+\epsilon+\nu}{2}}{\frac{1-s+\epsilon-\nu}{2}}{\frac{1}{2}+\epsilon}{-w^2}y^{s-\frac12}ds+H(z)\\
=&\frac{(2w)^{\epsilon}}{2\pi i}\int_{(c)}(-1)^{\epsilon}N ^{s-\frac12}\Lambda_g(s)(1+w^2)^{s-\frac12}\hyp{\frac{s+\epsilon+\nu}{2}}{\frac{s+\epsilon-\nu}{2}}{\frac{1}{2}+\epsilon}{-w^2}y^{s-\frac12}ds+H(z)\\
=&\frac{(-2w)^{\epsilon}}{2\pi i}\int_{(c)}\Lambda_g(s)\hyp{\frac{s+\epsilon+\nu}{2}}{\frac{s+\epsilon-\nu}{2}}{\frac{1}{2}+\epsilon}{-w^2}(N(1+w^2)y)^{s-\frac12}ds+H(z)\\
=&\tilde{g}\left(-\frac{1}{N(wy+iy)}\right)+H(z).
\end{align*}

If $\epsilon=1$, then this shows $f(z)=g(-1/Nz)$. Let $\epsilon=0$ and $\nu\neq 0$. The hypergeometric functions occuring in $H(z)$ can be evaluated directly\footnote{http://dlmf.nist.gov/15.4.6}:
\begin{align*}
\hyp{\frac12\pm\nu}{\frac12}{\frac12}{-w^2} =(1+w^2)^{-\frac12\mp\nu},~~
\hyp{\nu}{0}{\frac12}{-w^2}=\hyp{0}{-\nu}{\frac12}{-w^2} =1.
\end{align*}
So
\begin{multline*}
H(z) = ((1+w^2)y)^{-\frac12-\nu}\Res_{s=1+\nu}\Lambda_f(s)+
((1+w^2)y)^{-\frac12+\nu}\Res_{s=1-\nu}\Lambda_f(s)\\
+y^{\frac12-\nu}\Res_{s=\nu}\Lambda_f(s)
+y^{\frac12+\nu}\Res_{s=-\nu}\Lambda_f(s)\\
=\Res_{s=\nu}\Lambda_f(s)y^{\frac12-\nu}+\Res_{s=-\nu}\Lambda_f(s)y^{\frac12+\nu}+N^{\frac12+\nu}\Res_{s=1+\nu}\Lambda_g(s)\Im\left(-\frac{1}{Nz}\right)^{\frac12+\nu} \\
+ N^{\frac12-\nu}\Res_{s=1-\nu}\Lambda_g(s)\Im\left(-\frac{1}{Nz}\right)^{\frac12-\nu}.
\end{multline*}
It follows that $f(z) = g(-1/Nz)$ in this case. Now let $\epsilon=0$ and $\nu=0$. The first term in $H(z)$ is
\begin{multline*}
\Res_{s=0}\hyp{\frac{s}{2}}{\frac{s}{2}}{\frac12}{-w^2}\Lambda_f(s)y^{\frac12-s}=\lim_{s\rightarrow 0}\frac{d}{ds}\left(s^2\hyp{\frac{s}{2}}{\frac{s}{2}}{\frac12}{-w^2}\Lambda_f(s)y^{\frac12-s}\right)\\
=\lim_{s\rightarrow0}\frac{d}{ds}\hyp{\frac{s}{2}}{\frac{s}{2}}{\frac12}{-w^2}y^{\frac12-s}\left(a_0'-s a_0+s^2E_0(s)\right),
\end{multline*}
where $E_0(s)$ is holomorphic at $s=0$. To evaluate the limit, expand the hypergeometric function around $s=0$:
\begin{equation}
\hyp{\frac{s}{2}}{\frac{s}{2}}{\frac12}{-w^2}=\sum_{n=0}^{\infty}\frac{(-w^2)^n}{\left(\frac12\right)_nn!}\left(\frac{s}{2}\right)^2_n
=1+s^2F_0(s),
\end{equation}
where $F_0(s)$ is holomorphic at $s=0$. We see that
\begin{equation}\label{eq:res0}
\Res_{s=0}\hyp{\frac{s}{2}}{\frac{s}{2}}{\frac12}{-w^2}\Lambda_f(s)y^{\frac12-s}=\Res_{s=0}\Lambda_f(s) y^{\frac12}-\Res_{s=0}s\Lambda_f(s) y^{\frac12}\log y.
\end{equation}
Using the \eqref{eqn:euler identity} and the functional equation of $\Lambda_f(s)$
\begin{multline*}
\Res_{s=1}\hyp{\frac{s}{2}}{\frac{s}{2}}{\frac12}{-w^2}\Lambda_f(s)y^{\frac12-s}\\
=\Res_{s=1}\left((1+w^2)Ny\right)^{\frac12-s}\hyp{\frac{1-s}{2}}{\frac{1-s}{2}}{\frac12}{-w^2}\Lambda_g(1-s)\\
=-\Res_{s=0}\left((1+w^2)Ny\right)^{s-\frac12}\hyp{\frac{s}{2}}{\frac{s}{2}}{\frac12}{-w^2}\Lambda_g(s)\\
=\Im\left(-\frac{1}{Nz}\right)^{\frac12}\left(-\Res_{s=0}\Lambda_g(s) + \Res_{s=0}s\Lambda_g(s)\log\left(\Im\left(-\frac{1}{Nz}\right)\right) \right).
\end{multline*}
As before we conclude $f(z)=g(-1/Nz)$.

\subsection{Transformation properties from twisted functional equations}\label{Twists.section}
In this section we deduce further transformation properties of $f$ and $g$ from the twisted functional equations. For a primitive Dirichlet character $\psi$ modulo a prime $q$ we write 
\[
f_\psi(z) = \sum_{n\neq 0}\frac{\psi(n)a_n}{2\sqrt{|n|}}W_{\nu}(ny)e(nx)
\]
and define $g_\psi$ analogously. With the definition of Proposition \ref{prop:Mellin-transform} we have $\Lambda_f(s,\psi) = \Lambda_{f_\psi}(s)$.
\begin{lemma}\label{lem:difference}Let $q\in\mathcal{P}$ and let $\alpha = a/q,\beta = b/q$ and let $z=wy+iy\in\HH$. With the assumptions of Theorem~\ref{thm:EntireConverse}, one has
\begin{equation}\label{eqn:difference}
f(z+\alpha)-f(z+\beta) = \frac{\chi(q)}{q-1}
\sum_{\substack{\psi~(\modulo q)\\ \psi\neq\psi_0}}  \psi(-N)\left(\psi(a)-\psi(b)\right)\tau(\psi)g_{\overline{\psi}}\left(-\frac{1}{Nq^2 z}\right).
\end{equation}
\end{lemma}
\begin{proof}
As $\alpha,\beta\in\mathbb{R}$, we have
\begin{equation}\label{eq:diff}
f(z+\alpha)-f(z+\beta)=\tilde{f}(z+\alpha)-\tilde{f}(z+\beta)\\
\end{equation}
Applying the inverse Mellin transform to Proposition~\ref{prop:Mellin-transform}
\begin{align}\label{eq:inverse mellin transform twisted}
\begin{split}f(z+\alpha)-f(z+\beta)&=\sum_{j\in\{0,1\}}i^{-j}(2 w)^{[j+\epsilon]}\frac{1}{2\pi i}\int_{(c)}\left(\Lambda_f\left(s,\alpha,\cos^{(j)}\right)-\Lambda_f\left(s,\beta,\cos^{(j)}\right)\right)\\
&\qquad\qquad\qquad\qquad\qquad\qquad\qquad\cdot\hyp{\frac{s+\nu+[j+\epsilon]}{2}}{\frac{s-\nu+[j+\epsilon]}{2}}{\frac12+[j+\epsilon]}{-w^2}y^{\frac12-s}ds.
\end{split}\end{align}
By Lemma \ref{lem:Lcos and Lsin via twists}, $\Lambda_f\left(s,\alpha,\cos^{(j)}\right)-\Lambda_f\left(s,\beta,\cos^{(j)}\right)$ is a linear combination of twists of $\Lambda_f(s)$ by characters of conductor $q$. Therefore, by the assumptions of Theorem \ref{thm:EntireConverse}, it is entire and we can shift the path of integration to the left. The integral in equation \eqref{eq:inverse mellin transform twisted} equals
\begin{align*}
&\int_{(1-c)}\left(\Lambda_f\left(s,\alpha,\cos^{(j)}\right)-\Lambda_f\left(s,\beta,\cos^{(j)}\right)\right)\hyp{\frac{s+\nu+[j+\epsilon]}{2}}{\frac{s-\nu+[j+\epsilon]}{2}}{\frac12+[j+\epsilon]}{-w^2}y^{\frac12-s}ds\\
&\quad=\int_{(c)}\left(\Lambda_f\left(1-s,\alpha,\cos^{(j)}\right)-\Lambda_f\left(1-s,\beta,\cos^{(j)}\right)\right)\hyp{\frac{1-s+\nu+[j+\epsilon]}{2}}{\frac{1-s-\nu+[j+\epsilon]}{2}}{\frac12+[j+\epsilon]}{-w^2}y^{s-\frac12}ds\\
&\quad=(-1)^\epsilon\int_{(c)}\frac{i^j(q^2N)^{s-\frac12}\chi(q)}{q-1}
\sum_{\substack{\psi~(\text{mod }q)\\\psi\neq\psi_0\\\psi(-1)=(-1)^j}}
\psi(N)(\psi(a)-\psi(b))\tau(\psi)\Lambda_g(s,\overline{\psi})\\
&\hspace*{8cm} \cdot\hyp{\frac{1-s+\nu+[j+\epsilon]}{2}}{\frac{1-s-\nu+[j+\epsilon]}{2}}{\frac12+[j+\epsilon]}{-w^2}y^{s-\frac12}ds
.
\end{align*}
The last line follows from the functional equation in Proposition \ref{prop: SinCosFE}. Applying the Euler identity \eqref{eqn:euler identity} we see that $f(z+\alpha)-f(z+\beta)$ equals
\begin{align*}
&(-1)^\epsilon\frac{\chi(q)}{q-1}\sum_{j\in\{0,1\}}\sum_{\substack{\psi~(\text{mod }q)\\\psi\neq\psi_0\\\psi(-1)=(-1)^j}}(2 w)^{[j+\epsilon]}
\psi(N)\tau(\psi)(\psi(a)-\psi(b))
\\
&\qquad\cdot\frac{1}{2\pi i}\int_{(c)}\Lambda_g(s,\overline{\psi})
\hyp{\frac{s+\nu+[j+\epsilon]}{2}}{\frac{s-\nu+[j+\epsilon]}{2}}{\frac12+[j+\epsilon]}{-w^2}\left(\frac{1}{(1+w^2)Nq^2y}\right)^{\frac12-s}ds\\
&=\frac{\chi(q)}{q-1}\sum_{\substack{\psi~(\modulo 	q)\\\psi\neq\psi_0}} \psi(-N)\left(\psi(a)-\psi(b)\right)\tau(\psi)g_{\overline{\psi}}\left(-\frac{1}{Nq^2 z}\right).
\end{align*}
In the last line we applied the inverse Mellin transform to the equality in Proposition \ref{prop:Mellin-transform} for $h=g_{\overline{\psi}}$.
\end{proof}
\begin{proposition}\label{prop:transformation of twists} 
If $a_n$ (resp. $b_n$) satisfy the assumptions of Theorem~\ref{thm:EntireConverse}, then,
\[
f_\psi(z)=\chi(q)\psi(-N)\frac{\tau(\psi)}{\tau(\overline{\psi})}g_{\overline{\psi}}\left(-\frac{1}{Nq^2 z}\right)
\]
for all non-principal characters $\psi$ modulo $q$.
\end{proposition}
\begin{proof}
By the proof of Lemma \ref{lem:Lcos and Lsin via twists},
\[
f(z+\alpha)-f(z+\beta) = \frac{1}{q-1}\sum_{\substack{\psi~(\modulo q)\\\psi\neq \psi_0}}\tau(\overline{\psi})\left(\psi(a)-\psi(b)\right)f_\psi(z).
\]
So equation \eqref{eqn:difference} can be rearranged to
\begin{multline*}
\sum_{\substack{\psi~(\modulo q)\\\psi\neq \psi_0}}\psi(a)\left(\tau(\overline{\psi})f_\psi(z)-\chi(q)\psi(-N)\tau(\psi)g_{\overline{\psi}}\left(-\frac{1}{Nq^2 z}\right)\right)
\\=
\sum_{\substack{\psi~(\modulo q)\\\psi\neq \psi_0}}\psi(b)\left(\tau(\overline{\psi})f_\psi(z)-\chi(q)\psi(-N)\tau(\psi)g_{\overline{\psi}}\left(-\frac{1}{Nq^2 z}\right)\right).
\end{multline*}
This implies that the expression on the left-hand side is independent of the choice of $a\not\equiv 0\mod q$. In other words, it is a linear combination of non-principal characters modulo $q$ that produces a multiple of the principal character. Since the set of all characters modulo $q$ is linearly independent, the coefficients of this linear combination must vanish. 
\end{proof}
We continue along the lines of \cite[Section 1.5]{AFAR}. For $r\in\mathbb{Q}$, let $T^r=\sabcd1{r}01$ so that $\left(T^r\right)^{-1}=T^{-r}$. For a primitive Dirichlet character modulo $q\in \mathcal{P}$ we have
\begin{equation}\label{eq.fpsisuma}
f_\psi = \tau(\overline{\psi})^{-1}\sum_{a\mod q}\overline{\psi}(a)f|T^{a/q}.
\end{equation}
In Section~\ref{subsection:transformation props from FE}, we showed that $f|\sabcd 01{-N}0 = g$.  Multiplying equation~\eqref{eq.fpsisuma} by $\sabcd01{-Nq^2}0$,  we arrive at:
\begin{align}\label{eq.taufpsi}
\tau(\overline{\psi})f_\psi\left(-\frac{1}{Nq^2z}\right) &= \sum_{a\mod q}\overline{\psi(a)}f|\begin{pmatrix}
-Naq & 1 \\ -Nq^2 & 0
\end{pmatrix}= \sum_{a\mod q}\overline{\psi(a)}g|\begin{pmatrix}
q^2 & 0 \\ -Naq & 1
\end{pmatrix}\\\label{line2}
&=\sum_{a\mod q}\overline{\psi(\tilde{a})}g|\begin{pmatrix}
q^2 & 0 \\ -N\tilde{a}q & 1
\end{pmatrix}=\psi(-N)\sum_{a\mod q}\psi(a)g|\begin{pmatrix}
q & -a \\ -N\tilde{a} & \frac{Na\tilde{a}+1}{q}
\end{pmatrix}T^{a/q}.
\end{align}
In~\eqref{line2}, we replaced $a$ by $\tilde{a}$, an integer that is inverse to $-Na\mod q$, and used that $\overline{\psi(a)}=\psi(-N)\psi(\tilde{a})$. Proposition \ref{prop:transformation of twists} states that the left-hand side of equation~\eqref{eq.taufpsi} is equal to $\chi(q)\psi(-N)\tau(\psi)g_{\overline{\psi}}$. So
\begin{align}\label{eq:average transformation equation}
\chi(q)
\sum_{a\mod q}\psi(a)g|T^{a/q}
=
\sum_{a\mod q}\psi(a)g|\begin{pmatrix}
q & -a \\ -N\tilde{a} & \frac{Na\tilde{a}+1}{q}
\end{pmatrix}T^{a/q}.
\end{align}
Equation~\eqref{eq:average transformation equation} is true for all primitive characters $\psi$ modulo $q$. Taking linear combinations we see that we can replace $\psi(a)$ above with any function $c$ on $(\Z/q\Z)^\times$ that satisfies $\sum_{a\mod q}c(a)=0$. Now we choose another prime $s\in\mathcal{P}$, not equal to $q$, with $qs = 1+r\tilde{r}N$ for integers $r,\tilde{r}$. Let $c$ be the function which is $1$ at $r\bmod q$, $-1$ at $-r\bmod q$, and $0$ elsewhere. Replacing $\psi$ in  \eqref{eq:average transformation equation} with $c$:
\begin{align}\label{eq:replace psi with c}
\chi(q)
\left(g|T^{r/q} - g|T^{-r/q}\right)
=
g|\begin{pmatrix}
q & -r \\ -N\tilde{r} & s
\end{pmatrix}T^{r/q} - g|\begin{pmatrix}
q & r \\ N\tilde{r} & s
\end{pmatrix}T^{-r/q}.
\end{align}
To ease the notation, we extend the action of $\GL_2^+(\R)$ on functions to the group algebra $\mathbb{C}[\GL_2^+(\R)]$ and introduce the right-ideal $\Omega=\{w\in \mathbb{C}[\GL_2^+(\R)]:~g|w = 0\}$. Let $A_{\pm}=\sabcd{q}{\pm r}{\pm N\tilde{r}}{s}$. Equation \eqref{eq:replace psi with c} translates to
\begin{equation}\label{eq.qs}
\left(A_+
 - \chi(q)\right)T^{-r/q} \equiv \left(A_- - \chi(q)\right)T^{r/q}\mod \Omega.
\end{equation}
Note that $A_{\pm}^{-1}=\sabcd{s}{\mp r}{\mp N\tilde{r}}{q}$, so reversing the roles of $q$ and $s$ in the above analysis gives:
\begin{equation}\label{eq.sq}
\left(A_-^{-1}
 - \chi(s)\right)T^{-r/s}\equiv\left(A_+^{-1} - \chi(s)\right)T^{r/s}\mod\Omega.
\end{equation}
Multiplying equation~\eqref{eq.qs} by $T^{r/q}$ we obtain
\begin{equation}\label{eq.qs multiplied}
A_+ - \chi(q) \equiv  \left(A_- - \chi(q)\right)T^{2r/q}\mod\Omega.
\end{equation}
Multiplying equation~\eqref{eq.sq} by the matrix $-\chi(q)T^{r/s}A_-T^{2r/q}$, we see
\begin{align}\label{eq.sq multiplied}\begin{split}
\left(A_- - \chi(q)\right) T^{2r/q}
&\equiv  
-\chi(q)\left(A_+^{-1} - \chi(s)\right)T^{2r/s}A_-T^{2r/q}\mod\Omega\\
&\equiv\left(A_+-\chi(q)\right)A_+^{-1}T^{2r/s}A_-T^{2r/q}\mod\Omega.
\end{split}\end{align}
Comparing equations \eqref{eq.qs multiplied} and \eqref{eq.sq multiplied} we obtain
\begin{equation}\label{eq.1-M}
\left(A_+ - \chi(q)\right)(1-M(q,s,r))\in\Omega,
\end{equation}
where $M(q,s,r)=A_+^{-1}T^{2r/s}A_-T^{2r/q}$ or, more explicitly,
\begin{equation}\label{eq.Melliptic}
M(q,s,r)= \begin{pmatrix}
1 & \frac{2r}{q}\\\frac{-2N\tilde{r}}{s} & -3+\frac{4}{qs}
\end{pmatrix}.
\end{equation}
Recall that a matrix $M\in\SL_2(\mathbb{R})$ is elliptic if $|\text{tr}(M)|<2$. The matrix $M(q,s,r)$ in equation~\eqref{eq.Melliptic} is elliptic of infinite order, since its eigenvalues are not roots of unity. We sum up our construction in the following proposition.
\begin{proposition}\label{prop:constr of ell mat}
Let $q$ and $s$ be distinct elements of $\mathcal{P}$ with $qs = 1+r\tilde{r}N$, for $r,\tilde{r}\in\mathbb{Z}$. Then $M(q,s,r)$, defined as in \eqref{eq.Melliptic}, is an elliptic matrix of infinite order, with
\begin{equation}\label{eq:crucial M equation}
\left(g|\abcd qr{N\tilde{r}}s - \chi(q)g\right)\big|(1-M(q,s,r))=0.
\end{equation}
\end{proposition}
\subsection{Two Circles}\label{TwoCircles.subsection}
A matrix $M$ is elliptic if and only if $M$ has a unique fixed point in $\mathcal{H}$.  Weil's lemma states that a holomorphic function on the upper half-plane invariant under an infinite order elliptic matrix is constant \cite[Lemma~1.5.1]{AFAR}. We begin by proving a geometric interpretation of this statement.
\begin{lemma}
Let $h$ be a continuous function on $\HH$ that is invariant under an infinite order elliptic matrices $M$ with fixed point $z_0\in\HH$. Then $h$ is constant on the hyperbolic circles around $z_0$.
\end{lemma}
\begin{proof}
Let $K=\frac{1}{\sqrt{z_0-\overline{z_0}}}\sabcd{1}{-z_0}{1}{-\overline{z_0}}$ be the Cayley transform that maps the upper half plane $\mathcal{H}$ to the open unit disk $\mathcal{D}$ and takes $z_0\in\mathcal{H}$ to $0\in\mathcal{D}$. The transformation $L = KMK^{-1}$ on $\mathbb{P}^1(\mathbb{C})$ fixes $0$ and $\infty$ and hence has the form $\sabcd{e^{i\pi \theta}}{0}{0}{e^{-i\pi\theta}}$ with irrational $\theta$, because $M$ has infinite order. Consider the function $\widetilde{h}(z) = h(K^{-1}z)$. Since
\[
\widetilde{h}(Lz) = h(K^{-1}Lz) = h(K^{-1}LKK^{-1}z) = h(K^{-1}z) = \widetilde{h}(z),
\]
we get $\widetilde{h}(e^{2\pi i m \theta}z)=\widetilde{h}(z)$ for all $m\in\Z$. Since the set $\{e^{2\pi i m \theta}|m\in\Z\}$ is dense on the unit circle we get that $\widetilde{h}(z)=\widetilde{h}(|z|)$ and hence $\widetilde{h}$ is constant on all circles $C_r=\{re^{it}|t\in[0,2\pi)\}$. This implies that $h$ is constant on the preimages under the Cayley transform of the circles $C_r$. These are exactly the hyperbolic circles around $z_0$.
\end{proof}
A corollary of this is that if $h$ is holomorphic, then it is constant on all of $\mathcal{H}$. There are counter-examples to this in the real-analytic setting of Maass forms (cf. \cite[Section~3.4]{GM}). Instead, we resort to the following theorem.

\begin{theorem}\label{TwoCircles.theorem}
If $h$ is a continuous function on $\HH$ that is invariant under two infinite order elliptic matrices with distinct fixed points in $\HH$ then it is constant.
\end{theorem}
\begin{proof}
Let $M_1,M_2$ be the two elliptic matrices and $z_1,z_2$ be their fixed points in $\HH$. Let $K$ be the Cayley transform that maps $z_1$ to $0$. By the above proof $\widetilde{h}(z)=h(K^{-1}z)$ is constant on all circles around $0$.

Let $d= d_{\text{hyp}}$ be the hyperbolic distance on $\mathcal{D}$ and $y\in\mathcal{D}$. If $d(0,y) \leq d(0,Kz_2)$, then the circle of radius $d(0,y)$ around $0$ has two intersection points with the full hyperbolic line connecting $0$ and $Kz_2$. Since this line is a geodesic we can deduce the following. One of the intersection points, $y_1$, is between $0$ and $Kz_2$, satisfying $d(y_1, Kz_2)\leq d(0,y_1)+ d(y_1, Kz_2)=d(0,Kz_2)$. The other one, $y_2$, satisfies $d(y_2,Kz_2)
= d(y_2, 0)+d(0,Kz_2)\geq d(0,Kz_2)$. By the intermediate value theorem the circle also contains an element $y_3$ with $d(y_3, Kz_2) = d(0,Kz_2)$. Since $\widetilde{h}$ is constant on circles around the origin we have $\widetilde{h}(y) = \widetilde{h}(y_3)$. Now $\widetilde{h}$ is also constant around hyperbolic circles with centre $Kz_2$. Since $d(y_3, Kz_2) = d(0,Kz_2)$, we obtain $\widetilde{h}(y_3) = \widetilde{h}(0)$. This implies that $\widetilde{h}$ is constant on the closed disc with centre $0$ of radius $d(0,Kz_2)$.

Now suppose $\widetilde{h}$ is constant on the closed disc of radius $r\geq d(0,Kz_2)$ around the origin. The disc contains a point $y$ with $d(y,Kz_2) = r+d(0,Kz_2)$ and the point $Kz_2$ with $d(Kz_2,Kz_2)=0$. Hence $\widetilde{h}$ is also constant on the closed disc of radius $r+d(0,Kz_2)$ around $Kz_2$. Using that $\widetilde{h}$ is constant on circles around $0$ we get that $\widetilde{h}$ is constant on the disk arount the origin with radius $r+2d(0,Kz_2)$. Repeating this process we see that $\widetilde{h}$ is constant.
\end{proof}

\subsection{Proof of Theorem~\ref{thm:EntireConverse}}\label{section:entire}
We now combine the results of the previous sections to deduce Theorem \ref{thm:EntireConverse}. The essential point is the construction of two infinite order elliptic matrices.
\begin{proof}
Let $\gamma = \sabcd ab{Nc}d \in \Gamma_0(N)$. The proof will follow if we can find two distinct infinite order elliptic matrices under which $g|\sabcd{a}{b}{Nc}{d} - \chi(a) g$ is invariant. Let $q,s$ be two odd primes in $\mathcal{P}$ such that $q\equiv a$ and $s\equiv d$ modulo $Nc$, i.e., $q= a- uNc$ and $s= d-vNc$ for $u,v\in\Z$. By the assumptions on $\mathcal{P}$ there are infinitely many such $q$ and $s$. Let $r=b-av+uvNc-ud$.
Then
\[
g|\abcd{a}{b}{Nc}{d} = g|\abcd 1u01 \abcd qr{Nc}s \abcd 1v01 = g|\abcd qr{Nc}s \abcd 1v01.
\]
By Proposition \ref{prop:constr of ell mat} the function $g_1 = g|\sabcd qr{Nc}s - \chi(a)g$ is invariant under the elliptic matrix
\[
M(q,s,r) = \left(\begin{array}{cc}
1 & \frac{2 \, r}{q} \\
-\frac{2 \, N c}{s} & \frac{4}{q s} - 3
\end{array}\right).
\]
The fixed point of $M(q,s,r)$ in $\HH$ is given by
\[
z_1 = i \, \sqrt{-\frac{s^{2} {\left(\frac{1}{q s} - 1\right)}^{2}}{N^{2}
c^{2}} + \frac{4 \, r s}{q N c}} - \frac{s {\left(\frac{1}{q s} -
1\right)}}{N c}.
\]
Let $q'=a-u'Nc\equiv a\mod N$ be another prime in  $\mathcal{P}$, different to $q$ and $s$ and $r'=-av+u'vNc+b-u'd$. We have
\[
g|\abcd {q'}{r'}{Nc}s =  g|\abcd{a}{b}{Nc}{d}\abcd 1{-v}01 = 
g|\abcd {q}{r}{Nc}s.
\]
As above we can show that $g_1' = g|\sabcd {q'}{r'}{Nc}s-\chi(a)g = g_1$ is invariant under $M(q',s,r')$. 
The fixed point of $M_{q',s}$ is
\[
z_2 = i \, \sqrt{-\frac{s^{2} {\left(\frac{1}{q' s} - 1\right)}^{2}}{N^{2}
c^{2}} + \frac{4 \, r' s}{q' N c}} - \frac{s {\left(\frac{1}{q' s} -
1\right)}}{N c}.
\]
Comparing real parts we find $z_1\neq z_2$ if $q\neq q'$. Hence, by Theorem \ref{TwoCircles.theorem}, $g_1$ is constant. Since both $g|\sabcd qr{Nc}s$ and $g$ are eigenfunctions of the Laplace Beltrami operator of the same positive eigenvalue and constant functions have eigenvalue $0$, $g_1$ can only be the zero function.

Since $g_1=0,$ we have
\[
g|\abcd{a}{b}{Nc}{d}= g|\abcd qr{Nc}s \abcd 1v01 = \chi(a)g|\abcd 1v01 = \chi(a)g.
\]
\end{proof}

\section{Proof of Theorem~\ref{thm:PolarConverse}}\label{PolarConverse.subsection}

In this section we will prove Theorem~\ref{thm:PolarConverse} by showing that its assumptions imply those of Theorem~\ref{thm:EntireConverse}. To that end, let $f$ and $g$ be as in Theorem~\ref{thm:PolarConverse}. From now on we assume that $\nu\neq0$, though this can be relaxed with technical modifications.
\begin{lemma} 
For $w\in\mathbb{R}_{\geq0}$, the function $\Lambda_f(s)\hyp{\frac{s+\epsilon+\nu}{2}}{\frac{s+\epsilon-\nu}{2}}{\frac{1}{2}+\epsilon}{-w^2}$ is $O\left((\Im s)^{-M}\right)$ uniformly in $\Re s$ for every $M>0$ as $\Im{s}\rightarrow\infty$. Its poles lie in the strip $-\sigma\leq \Re s\leq \sigma+1$. It is uniformly bounded on every vertical strip outside of a small neighbourhood around every pole.
\end{lemma}
\begin{proof}
From Appendix~\ref{Hypergeometric.subsection}, it follows that  $\Lambda(f,s)\hyp{\frac{s+\epsilon+\nu}{2}}{\frac{s+\epsilon-\nu}{2}}{\frac{1}{2}+\epsilon}{-w^2}$ decays exponentially for fixed real part $\Re s>\sigma+1$ and $\Im s \to\infty$. The functional equation \eqref{eq:FE} and the Euler identity \eqref{eqn:euler identity} imply
\[
\Lambda_f(s)\hyp{\frac{s+\epsilon+\nu}{2}}{\frac{s+\epsilon-\nu}{2}}{\frac{1}{2}+\epsilon}{-w^2}=(-1)^\epsilon (N(1+w^2))^{\frac12-s}\Lambda_g(1-s)\hyp{\frac{1-s+\epsilon+\nu}{2}}{\frac{1-s+\epsilon-\nu}{2}}{\frac{1}{2}+\epsilon}{-w^2}.
\]
Therefore, we have exponential decay for $\Re s<-\sigma$. The occuring hypergeometric functions are entire in $s$. For $P(s)\in\mathbb{C}[s]$ as in Theorem~\ref{thm:PolarConverse}, the function $P(s)\Lambda(f,s)$ is entire, of finite order and bounded polynomially on vertical lines that lie outside of the critical strip. By the Phragm\'{e}n-Lindel\"of principle it is also uniformly bounded polynomially inside the critical strip, and hence $P(s)P(1-s)\Lambda(f,s)\hyp{\frac{s+\epsilon+\nu}{2}}{\frac{s+\epsilon-\nu}{2}}{\frac{1}{2}+\epsilon}{-w^2}$ decays exponentially on these strips.
\end{proof}
\begin{lemma}\label{CircleIntegral.lemma}
Integrating over any circle enclosing all poles of $\Lambda_f(s)$ we have, for any $z\in\mathcal{H}$,
\[
 \tilde{f}(z) - \tilde{g}\left(-\frac{1}{Nz}\right) 
 =\frac{(2w)^{\epsilon}}{2\pi i}\oint\Lambda_f(s)\hyp{\frac{s+\epsilon+\nu}{2}}{\frac{s+\epsilon-\nu}{2}}{\frac12+\epsilon}{-w^2}y^{\frac{1}{2}-s}ds,
\]
where $w=\Re z/\Im z$.
\end{lemma}
\begin{proof}
Let $z = wy+iy$ and $c>\sigma$. Applying the inverse Mellin transform to Proposition~\ref{prop:Mellin-transform} with $h=\tilde{f}$ and $\alpha=0$ implies
\[
\tilde{f}(z)=\frac{(2w)^{\epsilon}}{2\pi i}\int_{(c)}\Lambda_f(s)\hyp{\frac{s+\nu+\epsilon}{2}}{\frac{s-\nu+\epsilon}{2}}{\frac12+\epsilon}{-w^2}y^{\frac12-s}.
\]
Similarly, making the change of variables $s\to 1-s$ and applying the functional equation for $\Lambda_f$ along with the Euler identity, we see that
\begin{align*}
\tilde{g}\left(-\frac{1}{Nz}\right)&=
\frac{(2w)^{\epsilon}}{2\pi i}\int_{(c)}\Lambda_g(s)\hyp{\frac{s+\nu+\epsilon}{2}}{\frac{s-\nu+\epsilon}{2}}{\frac12+\epsilon}{-w^2}(Ny(1+w^2))^{s-\frac12}ds
\\&=\frac{(2w)^{\epsilon}}{2\pi i}\int_{(1-c)}\Lambda_f(s)\hyp{\frac{s+\epsilon+\nu}{2}}{\frac{s+\epsilon-\nu}{2}}{\frac12+\epsilon}{-w^2}y^{\frac12-s}ds.
\end{align*}
Since the integrand is meromorphic and rapidly decaying in $s$ we can write
\begin{align*}
\tilde{f}(z) - \tilde{g}\left(-\frac{1}{Nz}\right) &= \frac{(2w)^\epsilon}{2\pi i}\left(\int_{(c)}-\int_{(1-c)}\right)\Lambda_f(s)\hyp{\frac{s+\epsilon+\nu}{2}}{\frac{s+\epsilon-\nu}{2}}{\frac12+\epsilon}{-w^2}y^{\frac{1}{2}-s}ds\\
&= \frac{(2w)^\epsilon}{2\pi i}\oint\Lambda_f(s)\hyp{\frac{s+\epsilon+\nu}{2}}{\frac{s+\epsilon-\nu}{2}}{\frac12+\epsilon}{-w^2}y^{\frac{1}{2}-s}ds.
\end{align*}
\end{proof}
Let $\alpha\in\Q_{>0}$ and $z=\alpha(1+iy)$ for $y>0$. The above Lemma now reads
\begin{equation}\label{eq:circleintegralspecial}
 \tilde{f}(z) - \tilde{g}\left(-\frac{1}{Nz}\right) 
 =\frac{(2y^{-1})^\epsilon}{2\pi i}\oint\Lambda_f(s)\hyp{\frac{s+\epsilon+\nu}{2}}{\frac{s+\epsilon-\nu}{2}}{\frac12+\epsilon}{-y^{-2}}(\alpha y)^{\frac{1}{2}-s}ds.
\end{equation}
From now on we let $\beta=-1/N\alpha$.
\begin{lemma}[Lemma 2.4 in \cite{BCK}]\label{lem:g approximation} 
Let $y\in(0,\frac12]$. With the assumptions of Theorem~\ref{thm:PolarConverse}, for any $\ell_0\geq 0$
\begin{multline*}
\tilde{g}\left(-\frac{1}{Nz}\right) =O_{\sigma,\alpha,\ell_0}(y^{2\ell_0-\sigma})+
\sum_{a\in\{0,1\}}i^{-a}
\sum_{\substack{t=0\\t\equiv a+\epsilon\bmod2}}^{2\ell_0-1}
\frac{(2\pi i N \alpha)^t}{t!}\\
\cdot\frac{1}{2\pi i}\int_{(\sigma+1)}
\Lambda_g\left(s+t,\beta,\cos^{(a)}\right)\frac
{
\gamma_f^{(-)^{\epsilon}}(1-s)}
{\gamma_f^{(-)^{\epsilon}}(1-s-2\lfloor t/2\rfloor)}\left(\frac{y}{N\alpha}\right)^{1/2-s}ds.
\end{multline*}
\end{lemma}
\begin{remark}
Although Lemma 2.4 in \cite{BCK} is stated for a function $F$ that is related to a Maass form, an inspection of the proof shows that we can still apply it to $g$.
\end{remark}
Suppose that $0<y<1$. By \eqref{eq:hypergeometric continuation}
\begin{multline}\label{eqn:hypergeometricoutsidedisc}
(2y^{-1})^{\epsilon}\hyp{\frac{s+\epsilon+\nu}{2}}{\frac{s+\epsilon-\nu}{2}}{\frac12+\epsilon}{-y^{-2}}(\alpha y)^{\frac12-s}\\
=\sum_{\pm}\frac{\Gamma(\mp\nu)\pi^{\epsilon+\frac12}\alpha^{\frac12-s}}{\Gamma\left(\frac{s+\epsilon\mp\nu}{2}\right)\Gamma\left(\frac{1-s+\epsilon\mp\nu}{2}\right)}\sum_{k=0}^{\ell_0-1}\frac{\left(\frac{s+\epsilon\pm\nu}{2}\right)_k\left(\frac{s-\epsilon\pm\nu+1}{2}\right)_k}{k!(1\pm\nu)_k}(-1)^k y^{2k\pm\nu+\frac12}+O_{\ell_0,\nu,s}\left(y^{2\ell_0}\right).
\end{multline}
Moreover, if $s$ is in a fixed compact set then we can choose the error terms to be independent of $s$. For $\ell_0\geq 0$, we therefore conclude, for any $0<y<1$,
\begin{multline}\label{eqn:Ik}
\frac{(2y^{-1})^{\epsilon}}{2\pi i}\oint\Lambda_f(s)\hyp{\frac{s+\epsilon+\nu}{2}}{\frac{s+\epsilon-\nu}{2}}{\frac12+\epsilon}{-y^{-2}}(\alpha y)^{\frac{1}{2}-s}ds \\
=\sum_{k=0}^{\ell_0-1}\sum_{\pm}\mathcal{I}^{\pm}_k(\alpha) y^{\pm\nu+2k+\frac12}+O_{\ell_0,\nu,\sigma}(y^{2\ell_0}),
\end{multline}
where the integral is taken over any circle containing all poles of $\Lambda_f(s)$ and
\begin{align}
\mathcal{I}^{\pm}_k(\alpha) &=(-1)^k\frac{\Gamma\left(\mp\nu\right)\sqrt{\pi}}{k! \left(1\pm\nu\right)_k}\cdot\frac{1}{2\pi i} \oint\Lambda_f(s)\frac{\left(\frac{s+\epsilon\pm\nu}{2}\right)_k \left(\frac{s-\epsilon\pm\nu+1}{2}\right)_k}
{\Gamma\left(\frac{s+\epsilon\mp\nu}{2}\right)\Gamma\left(\frac{1-s+\epsilon\mp\nu}{2}\right)}\alpha^{\frac12-s}ds.
\end{align}
Analoguously to \cite[Section~2]{WCTWP}, we make the following definition.
\begin{definition}
For $\nu\in\mathbb{C}\setminus\{0\}$ and any open interval $(a,b)\subset\mathbb{R}$, denote by $\mathcal{M}^\nu(a,b)$ the set of meromorphic functions which are holomorphic on $a\leq\Re(s)\leq b$, except for at most simple poles in the sets $\pm\nu+\Z$, and bounded on $\{s\in\mathbb{C}:\Re(s)\in[c,d],|\Im(s)|\geq1\}$ for each compact $[c,d]\subset(a,b)$. 
\end{definition}
We will also consider the following subsets of $\mathcal{M}^\nu(a,b)$: 
\begin{equation}
\mathcal{M}^{\nu}_{t}(a,b)=\{f\in\mathcal{M}^\nu(a,b):f\text{ holomorphic at }s\in2\mathbb{Z}+t+1\pm\nu\},~t\in\mathbb{Z},
\end{equation}
\begin{equation}
\mathcal{H}(a,b)=\{f\in\mathcal{M}^\nu(a,b):f\text{ holomorphic at }s\in\mathbb{Z}\pm\nu\}.
\end{equation}

\begin{lemma}\label{HolomorphicFunction.lemma}
For $\alpha\in\Q_{>0}$ the following function is in $\mathcal{H}(\sigma-2\ell_0,\infty)$,
\begin{multline}\label{eqn:holfun}
H_\alpha(s) = (N\alpha^2)^{s-\frac12}\sum_{a\in\{0,1\}}i^{-a}\sum_{\substack{t=0\\t\equiv a+\epsilon\bmod2}}^{2\ell_0-1}
\frac{(2\pi i N \alpha)^t}{t!}
\Lambda_g\left(s+t,\beta,\cos^{(a)}\right)\frac
{
\gamma_f^{(-)^{\epsilon}}(1-s)}
{\gamma_f^{(-)^{\epsilon}}(1-s-2\lfloor t/2\rfloor)}\\
-i^{-\epsilon}\pi^{\epsilon}\Lambda_f\left(s,\alpha,\cos^{(\epsilon)}\right)
+\alpha^{s-\frac12}\sum_{k=0}^{\ell_0-1}\left(\frac{\mathcal{I}^{+}_k(\alpha)}{s+\nu+2k}+\frac{\mathcal{I}^{-}_k(\alpha)}{s-\nu+2k}\right).
\end{multline}
\end{lemma}

\begin{proof}
Let $\chi_{(0,1)}$ be the characteristic function of the interval $(0,1)$. By Lemmas \ref{CircleIntegral.lemma} and \ref{lem:g approximation} we have
\begin{multline}
F_\alpha(y):=\sum_{a\in\{0,1\}}i^{-a}
\sum_{\substack{t=0\\t\equiv a+\epsilon\bmod2}}^{2\ell_0-1}
\frac{(2\pi i N \alpha)^t}{t!}\\
\cdot\frac{1}{2\pi i}\int_{(\sigma+1)}
\Lambda_g\left(s+t,\beta,\cos^{(a)}\right)\frac
{
\gamma_f^{(-)^{\epsilon}}(1-s)
}
{
\gamma_f^{(-)^{\epsilon}}(1-s-2\lfloor t/2\rfloor)
}\left(\frac{y}{N\alpha}\right)^{1/2-s}ds\\
 -\tilde{f}(\alpha+i\alpha y)+\chi(y)\frac{(2y^{-1})^{\epsilon}}{2\pi i}\oint\Lambda_f(s)\hyp{\frac{s+\epsilon+\nu}{2}}{\frac{s+\epsilon-\nu}{2}}{\frac12+\epsilon}{-y^{-2}}(\alpha y)^{\frac{1}{2}-s}ds\\
=O_{\nu,\sigma,\alpha,\ell_0}(y^{2\ell_0-\sigma}).
\end{multline}
Hence we have that the Mellin transform $\int_0^{\infty}\alpha^{s-\frac12}F_\alpha(y)y^{s-\frac12}\frac{dy}{y}$ is in $\mathcal{H}\left(\sigma-2\ell_0,\infty\right)$ and indeed this equals $H_\alpha(s)$. Each term in $H_\alpha(s)$ is the Mellin transform of the corresponding term in $\alpha^{s-\frac12}F_\alpha(y)$. The first follows from Mellin inversion, while for the second term we use Proposition \ref{prop:Mellin-transform}. The last term is
\begin{equation}
\alpha^{s-\frac12}\int_0^{\infty}\chi(y)\sum_{k=0}^{\ell_0-1}\sum_{\pm}\mathcal{I}^{\pm}_k(\alpha)y^{\pm\nu+2k+\frac12}y^{s-\frac12}\frac{dy}{y}
=\alpha^{s-\frac12}\sum_{k=0}^{\ell_0-1}\sum_{\pm}\frac{\mathcal{I}^{\pm}_k(\alpha)}{s\pm\nu+2k}.
\end{equation}
\end{proof}
Suppose $\beta=\frac{u}{v}\in\mathbb{Q}^{\times}$ with $(u,v)=1$ and $u>0$. Given $\mathcal{P}$ as in Theorem~\ref{thm:PolarConverse}, we introduce the infinite set
\[
T_\beta := \left\{ \frac{p}{u}\in\mathbb{Q}_{>0}: \, p\equiv u\mod v,~p\in \mathcal{P}\right\}.
\]
An important feature of the sets $T_\beta$ is that if $\lambda\in T_\beta$, then  $\Lambda_g\left(s,\lambda\beta,\cos^{(j)}\right)=\Lambda_g\left(s,\beta,\cos^{(j)}\right)$.

Consider $t_0\in\mathbb{Z}_{\geq0}$ and any subset $T_{\beta,M}\subset T_{\beta}$ of cardinality $M\geq 2\ell_0 > t_0$. For each $\lambda\in T_{\beta,M}$, since the Vandermonde determinant does not vanish, there exist $c_{\lambda}\in\mathbb{C}$ such that
\begin{equation}\label{Vandermonde}
\sum_{\lambda\in T_{\beta,M}}c_{\lambda}\lambda^{-t}=\delta_{t_0}(t),~t\in\{0,1,\cdots,2\ell_0-1\}.
\end{equation}

\begin{lemma}\label{lem:isolated term} 
Let $\alpha\in\Q_{>0}$, $t_0\in\mathbb{Z}_{\geq0}$, $T_{\beta,M}$ of size $M\geq 2\ell_0> t_0$, and $c_\lambda\in\Q$ be as in \eqref{Vandermonde}. The following function is in $\mathcal{H}(t_0+\sigma-2\ell_0,\infty)$:
\begin{multline}\label{eq:isolatedterm}
i^{-[\epsilon+t_0]}(N\alpha^2)^{s-\frac12}\alpha^{-t_0}
\frac{(2\pi i)^{t_0}}{t_0!}
\Lambda_g\left(s,\beta,\cos^{([\epsilon+t_0])}\right)\frac
{
\gamma_f^{(-)^{\epsilon}}(1-s+t_0)
}
{
\gamma_f^{(-)^{\epsilon}}(1-s+[t_0])
}\\
-\sum_{\lambda\in T_{\beta,M}}c_\lambda\lambda^{2s-2t_0-1}
\left(
(-i\pi)^{\epsilon}\Lambda_f\left(s-t_0,\alpha\lambda^{-1},\cos^{(\epsilon)}\right)-(\lambda^{-1}\alpha)^{s-t_0-\frac12}\sum_{k=0}^{\ell_0-1}
\sum_{\pm}\frac{\mathcal{I}^{\pm}_k(\alpha\lambda^{-1})}{s-t_0\pm\nu+2k}
\right).
\end{multline}
\end{lemma}
\begin{proof}
Recall that for every $\lambda\in T_{\beta,M}$ we have $\Lambda_g\left(s,\lambda\beta,\cos^{(j)}\right)=\Lambda_g\left(s,\beta,\cos^{(j)}\right)$. The function in \eqref{eq:isolatedterm} is $\sum_{\lambda\in T_{\beta,M}}c_\lambda\lambda^{2s-2t_0-1}H_{\lambda^{-1}\alpha}(s-t_0)$ and hence the statement follows from Lemma \ref{HolomorphicFunction.lemma}. We just note that an important step in the calculation is applying \eqref{Vandermonde} as follows:
\begin{multline}\label{eq:hol}
\sum_{\lambda\in T_{\beta,M}}c_{\lambda}\lambda^{2s-1}(N(\lambda^{-1}\alpha)^2)^{s-\frac12}\sum_{a\in\{0,1\}}i^{-a}\sum_{\substack{t=0\\t\equiv a+\epsilon\bmod2}}^{2\ell_0-1}
\frac{(2\pi i N\lambda^{-1}\alpha)^t}{t!}\\
\cdot
\Lambda_g\left(s+t,\lambda\beta,\cos^{(a)}\right)\frac
{
\gamma_f^{(-)^{\epsilon}}(1-s)}
{\gamma_f^{(-)^{\epsilon}}(1-s-2\lfloor t/2\rfloor)}\\
=i^{-[\epsilon+t_0]}\left(N\alpha^2\right)^{s-\frac12}
\frac{(2\pi i N\alpha)^{t_0}}{t_0!}
\Lambda_g\left(s+t_0,\beta,\cos^{([\epsilon+t_0])}\right)\frac
{
\gamma_f^{(-)^{\epsilon}}(1-s)}
{\gamma_f^{(-)^{\epsilon}}(1-s-2\lfloor t_0/2\rfloor)}.
\end{multline}
\end{proof}
In particular, the following function is in $\mathcal{M}^v(t_0+\sigma-M+2,\infty)$:
\begin{multline}\label{eq:merfn}
i^{-[\epsilon+t_0]}(N\alpha^2)^{s-\frac12}\alpha^{-t_0}
\frac{(2\pi i)^{t_0}}{t_0!}
\Lambda_g\left(s,\beta,\cos^{([\epsilon+t_0])}\right)\frac
{
\gamma_f^{(-)^{\epsilon}}(1-s+t_0)
}
{
\gamma_f^{(-)^{\epsilon}}(1-s+[t_0])
}\\
-(-i\pi)^\epsilon\sum_{\lambda\in T_{\beta,M}}c_\lambda\lambda^{2s-2t_0-1}
\Lambda_f\left(s-t_0,\alpha\lambda^{-1},\cos^{(\epsilon)}\right).
\end{multline}
 In fact, \eqref{eq:merfn} is in $\mathcal{M}^\nu_{t_0}(t_0+\sigma-M+2,\infty)$.
\begin{proposition}\label{AnalyticPropertiesTwists.proposition}
Let $q\in\mathcal{P}\cup\{1\}$ and let $b\in\mathbb{Z}$ be coprime to $Nq$. Under the assumptions of Theorem~\ref{thm:PolarConverse}, if $\beta=\frac{b}{Nq}$ then, for any $\delta\in\{0,1\}$, both $\Lambda_f\left(s,\beta,\cos^{(\delta)}\right)$ and $\Lambda_g\left(s,\beta,\cos^{(\delta)}\right)$ continue to elements of $\mathcal{M}^\nu\left(-\infty,\infty\right)$.
\end{proposition}
\begin{proof}
We will present the proof for $\Lambda_g\left(s,\beta,\cos^{(\delta)}\right)$. Reversing the roles of $f$ and $g$, one may recover the result for $\Lambda_f\left(s,\beta,\cos^{(\delta)}\right)$. Observe that we can replace $b$ with $-b'$, where $b'\in\mathcal{P}$ such that $b'\equiv -b\bmod Nq$. Indeed $\Lambda_f(s,\beta,\cos^{(\delta)}) = \Lambda_f(s,-\frac{\tilde{b'}}{Nq},\cos^{(\delta)})$. Therefore we may assume $\beta<0$, $\alpha=-1/N\beta>0$ and $-b\in\mathcal{P}$. 

Consider $q'\in\mathcal{P}$ such that $q'\neq q$ and $(b,q')=1$. Let $\beta = \frac{b}{Nq}$ and $\beta'=\frac{b}{Nq'}$ with $b\in\mathcal{P}$. As $\beta$ and $\beta'$ have the same numerator, $T_{\beta}\cap T_{\beta'}$ is infinite. Let  $t_0\in\Z_{\geq 0}$ and choose a subset $T_M\subset T_{\beta}\cap T_{\beta'}$ of cardinality $M>t_0$ and $c_\lambda$ such that \eqref{Vandermonde} is satisfied. By considering the difference of equation~\eqref{eq:merfn} evaluated at $\beta$ and $\beta'$, we see that the following function is in $\mathcal{M}^\nu_{t_0}(t_0+\sigma-M+2,\infty)$:
\begin{multline}\label{eq:diff2}
\frac
{
\gamma_f^{(-)^{\epsilon}}(1-s+t_0)
}
{
\gamma_f^{(-)^{\epsilon}}(1-s+[t_0])
}
\cdot\left(\alpha^{2s-t_0-1}\Lambda_g\left(s,\beta,\cos^{([\epsilon+t_0])}\right)-\alpha'^{2s-t_0-1}\Lambda_g\left(s,\beta',\cos^{([\epsilon+t_0])}\right)\right).\\
-\frac{i^{[\epsilon+t_0]}(-i\pi)^\epsilon t_0!}{N^{s-\frac12}(2\pi i)^{t_0}}\sum_{\lambda\in T_M}c_\lambda \lambda^{2s-2t_0-1}
\left(\Lambda_f\left(s-t_0,\alpha\lambda^{-1},\cos^{(\epsilon)}\right)-\Lambda_f\left(s-t_0,\alpha'\lambda^{-1},\cos^{(\epsilon)}\right)\right)\\
\end{multline}
The assumption that $a_n=O(n^{\sigma})$ and the fact that the poles of $\gamma_f^{\pm}(s)$ lie in the half plane $\Re s < |\nu| < 1$ imply that, for all $\lambda\in T_M$, $\Lambda_f\left(s,\alpha\lambda^{-1},\cos^{(n)}\right)$ is holomorphic in the half plane $\Re s>\sigma+1$. The functional equation in Proposition \ref{prop: SinCosFE} hence implies that 
\begin{equation}\label{eqn:LambdaFDifference}
\Lambda_f\left(s-t_0,\alpha\lambda^{-1},\cos^{(\epsilon)}\right)-\Lambda_f\left(s-t_0,\alpha'\lambda^{-1},\cos^{(\epsilon)}\right)
\end{equation}
is in $\mathcal{H}(-\infty, t_0-\sigma)$. Note that in the case $\epsilon=1$, each of the terms in \eqref{eqn:LambdaFDifference} is in $\mathcal{H}(-\infty,t_0-\sigma)$ by Proposition \ref{prop: SinCosFE}. 

We deduce that, for every $t_0\in\mathbb{Z}_{\geq0}$, the following function is in $\mathcal{M}^\nu_{t_0}(t_0+\sigma-M+2, t_0-\sigma)$:
\begin{equation}\label{eq:gammadiff}
\frac
{\gamma_f^{(-)^{\epsilon}}(1-s+t_0)}
{\gamma_f^{(-)^{\epsilon}}(1-s+[t_0])}\left(\alpha^{2s-t_0-1}\Lambda_g\left(s,\beta,\cos^{([\epsilon+t_0])}\right)-\alpha'^{2s-t_0-1}\Lambda_g\left(s,\beta',\cos^{([\epsilon+t_0])}\right)\right).
\end{equation} 
Since the above does not depend on $T_M$ anymore, we can take $M$ arbitrarily large and see that \eqref{eq:gammadiff} is in $\mathcal{M}^\nu_{t_0}(-\infty, t_0-\sigma)$. The zeros of the quotient of gamma functions
\begin{equation}\label{eqn:gammaquotient}
\frac
{\gamma_f^{(-)^{\epsilon}}(1-s+t_0)}
{\gamma_f^{(-)^{\epsilon}}(1-s+[t_0])}.
\end{equation}
are contained in the set of poles of $\gamma_f^{(-1)^\epsilon}(1-s+[t_0])$ which is contained in $2\Z_{\geq 0}+1+[t_0]\pm\nu$. As $\nu\neq 0$, thepoles are all simple. Hence, dividing \eqref{eq:gammadiff} by \eqref{eqn:gammaquotient} we see that
\begin{equation}\label{eqn:difference of Lambda_g}
\alpha^{2s-t_0-1}\Lambda_g\left(s,\beta,\cos^{([\epsilon+t_0])}\right)-\alpha'^{2s-t_0-1}\Lambda_g\left(s,\beta',\cos^{([\epsilon+t_0])}\right),
\end{equation}
is in $\mathcal{M}^\nu(-\infty,t_0-\sigma)$. As $\alpha\neq\alpha'$ and we may take arbitrary $t_0\geq 0$, we conclude that $\Lambda_g\left(s,\beta,\cos^{(\epsilon+t_0)}\right)$ is in $\mathcal{M}^\nu(-\infty,t_0-\sigma)$. The function $\Lambda_g\left(s,\beta,\cos^{(\epsilon+t_0)}\right)$ only depends on the parity of $t_0$, so again we can choose $t_0$ arbitrarly large, but of a fixed parity, to conclude that $\Lambda_g\left(s,\beta,\cos^{(\epsilon+t_0)}\right)$ is in $\mathcal{M}^\nu(-\infty,\infty)$.
\end{proof}
\begin{corollary}\label{AnalyticPropertiesTwists2.proposition}
Make the assumptions of Proposition \ref{AnalyticPropertiesTwists.proposition}. If $\beta=b/q$ for $q\in \{1\}\cup\mathcal{P}$ and $(b,q)=1$, then, for any $\delta\in\{0,1\}$, $\Lambda_f\left(s,\beta,\cos^{(\delta)}\right)$ and $\Lambda_g\left(s,\beta,\cos^{(\delta)}\right)$ continue to elements of $\mathcal{M}^\nu\left(-\infty,\infty\right)$.
\end{corollary}
\begin{proof}
Let $\alpha=-\frac{1}{N\beta}=-\frac{q}{Nb}$. As in the proof of Proposition \ref{AnalyticProperties.proposition} we can make the assumption that $\beta<0$ and hence $\alpha>0$ and $-b\in\mathcal{P}$. Consider $t_0\in\mathbb{Z}_{\geq0}$, $M>t_0$ and $T_{\beta,M}$ a subset of $T_\beta$ of cardinality $M$ satisfying \eqref{Vandermonde}. For any $\lambda\in T_{\beta,M}$ we have that $\alpha\lambda^{-1}=-\frac{q}{Np}$ has the form required in Proposition \ref{AnalyticPropertiesTwists.proposition}, so $\Lambda_f\left(s-t_0,\alpha\lambda^{-1},\cos^{(\epsilon)}\right)$ is in $\mathcal{M}^\nu(-\infty,\infty)$. Then for $t_0\in\{0,1\}$, equation \eqref{eq:merfn} implies that $\Lambda_g\left(s,\beta,\cos^{([\epsilon+t_0])}\right)$ is in $\mathcal{M}^\nu(t_0+\sigma-M+2,\infty)$. Taking $M$ arbitrarily large, we see that both $\Lambda_g\left(s,\beta,\cos^{(0)}\right)$ and $\Lambda_g\left(s,\beta,\cos^{(1)}\right)$ are in $\mathcal{M}^\nu(-\infty,\infty)$. Reversing the roles of $f$ and $g$, we draw the same conclusion for $\Lambda_f\left(s,\beta,\cos^{(\delta)}\right)$.
\end{proof}
We now assume that $\sigma=\frac12+\kappa$, where $0<\kappa<\frac12-|\Re\nu|$. Under this assumption, $\Lambda_f(s)$ and $\Lambda_g(s)$ are holomorphic for $\Re s>1+|\Re\nu|$ and so equation~\ref{eq:FE} implies they are also holomorphic for $\Re s<-|\Re\nu|$. Corollary~\ref{AnalyticPropertiesTwists2.proposition} then implies that $\Lambda_f(s)$ and $\Lambda_g(s)$ are holomorphic away from the set $\{\pm\nu,1\pm\nu\}$, where they have at most simple poles if $\nu\neq0$. 

To finish the proof of Theorem~\ref{thm:PolarConverse} it thus suffices to show firstly that if $\epsilon=1$ then $\Lambda_f(s)$ and $\Lambda_g(s)$ are entire, and secondly that if $\psi$ is a primitive Dirichlet character with conductor $q\in \mathcal{P}$ then $\Lambda_f(s,\psi)$ and $\Lambda_g(s,\psi)$ are entire. 

\begin{lemma}\label{Residues.lemma}
Assume that $\alpha\in\mathbb{Q}_{>0}$ and $\beta=-\frac{1}{N\alpha}$ are such that $\Lambda_g\left(s,\beta,\cos^{(\delta)}\right)$ and $\Lambda_f\left(s,\alpha, \cos^{(\delta)}\right)$ continue to elements of $\mathcal{M}^\nu(-\infty,\infty)$ for $\delta\in\{0,1\}$.

If $s_0\in\mathbb{Z}$ satisfies $s_0<1$, we choose an integer $t_0\in\mathbb{Z}_{>1}$ such that $[t_0]=[s_0]$ and write $j=\frac12(t_0-s_0)$. Moreover, we choose   a set $T_{\beta,M}$ of size $M\geq 2\ell_0 >t_0-s_0=2j$ satisfying \eqref{Vandermonde}. Fix a sign $\delta\in\{\pm\}$. 

If $\epsilon=0$, then
\begin{multline}\label{eq:resnot0}
\sum_{\lambda\in T_{\beta,M}}c_\lambda\lambda^{2s_0-2t_0-1+2\delta\nu}
\Res_{s=s_0+\delta\nu}\Lambda_f\left(s-t_0,\alpha\lambda^{-1},\cos\right)\\
=i^{-[t_0]}(N\alpha^2)^{s_0-\frac12+\delta\nu}\alpha^{-t_0}
\frac{(2\pi i)^{t_0}}{t_0!}\frac
{
\gamma_f^{+}(1-s_0+t_0-\delta\nu)
}
{
\gamma_f^{+}(1-s_0+[t_0]-\delta\nu)
}\Res_{s=s_0+\delta\nu}
\Lambda_g\left(s,\beta,\cos^{([t_0])}\right)\\
+(-1)^j\delta_{0}(s_0)\frac{\left(\frac12-\delta\nu\right)_j}{j!}\alpha^{s_0-t_0-1+2\delta\nu}\Res_{s=1-\delta\nu}\Lambda_f(s).
\end{multline}
If $\epsilon=1$, then
\begin{multline}\label{eq:nunot0epsilon1}
i\pi\sum_{\lambda\in T_{\beta,M}}c_\lambda\lambda^{2s_0+2\delta\nu-2t_0-1}
\Res_{s=s_0+\delta\nu}\Lambda_f\left(s-t_0,\alpha\lambda^{-1},\sin\right)\\
=i^{-[1+t_0]}(N\alpha^2)^{s_0-\frac12+\delta\nu}\alpha^{-t_0}
\frac{(2\pi i)^{t_0}}{t_0!}\frac
{
\gamma_f^{-}(1-s_0+t_0-\delta\nu)
}
{
\gamma_f^{-}(1-s_0+[t_0]-\delta\nu)
}\Res_{s=s_0+\delta\nu}\Lambda_g\left(s,\beta,\cos^{([1+t_0])}\right)\\
+(-1)^j
\delta_{0}(s_0)\frac{\left(\frac12-\delta\nu\right)_j\Gamma(\delta\nu)\sqrt{\pi}}{j!\Gamma\left(\delta\nu+\frac12\right)}\alpha^{s_0-t_0+2\delta\nu-1}\Res_{s=1-\delta\nu}\Lambda_f(s)\\
+(-1)^j \sum_{\lambda\in T_{\beta,2\ell_0}}c_\lambda \lambda^{s_0-t_0+2\delta\nu}\frac{\left(\frac12\right)_j}{\delta\nu(1-\delta\nu)_j}\alpha^{s_0-t_0-1}\Res_{s=1+\delta\nu}\Lambda_f(s).
\end{multline}
\end{lemma}
\begin{proof}
As $0\leq j< \ell_0$, we have
\begin{equation}
\Res_{s=s_0+\delta\nu}\left(\sum_{k=0}^{\ell_0-1}\sum_{\pm}\frac{\mathcal{I}^{\pm}_k(\alpha\lambda^{-1})}{s-t_0\pm\nu+2k}\right)
=\mathcal{I}^{-\delta}_j(\alpha\lambda^{-1}).
\end{equation}

Since $\Re(s_0+\delta\nu) > t_0+\sigma-2\ell_0$, the residue of \eqref{eq:isolatedterm} at $s=s_0+\delta\nu$ is zero. Hence
\begin{align}\label{eqn:residue at s}
\begin{split}
\Res_{s=s_0+\delta\nu}&
\left(\sum_{\lambda\in T_{\beta,2\ell_0}}c_\lambda\lambda^{2s-2t_0-1}
(-i\pi)^{\epsilon}\Lambda_f\left(s-t_0,\alpha\lambda^{-1},\cos^{(\epsilon)}\right)\right)\\
&=\Res_{s=s_0+\delta\nu}\left(i^{-[\epsilon+t_0]}(N\alpha^2)^{s-\frac12}\alpha^{-t_0}
\frac{(2\pi i)^{t_0}}{t_0!}
\Lambda_g\left(s,\beta,\cos^{([\epsilon+t_0])}\right)\frac
{
\gamma_f^{(-)^{\epsilon}}(1-s+t_0)
}
{
\gamma_f^{(-)^{\epsilon}}(1-s+[t_0])
}\right)\\
&\quad+\sum_{\lambda\in T_{\beta,M}}c_\lambda (\lambda\alpha)^{s_0-t_0+\delta\nu-\frac12}
\mathcal{I}^{-\delta}_j(\alpha\lambda^{-1}).
\end{split}
\end{align}

Taking $b=q=1$ in Proposition \ref{AnalyticPropertiesTwists.proposition}, we see that $\Lambda_f(s),\Lambda_g(s)\in\mathcal{M}^\nu(-\infty,\infty)$. Since the poles of $\Lambda_g(s)$ are in the critical strip $-\sigma<\Re s< \sigma+1$, we have:
\begin{multline}
\mathcal{I}^{-\delta}_j(\alpha\lambda^{-1})=(-1)^j\frac{\Gamma\left(\delta\nu\right)\sqrt{\pi}}{j! \left(1-\delta\nu\right)_j}\cdot\frac{1}{2\pi i}\oint\Lambda_f(s)\frac{\left(\frac{s+\epsilon-\delta\nu}{2}\right)_j \left(\frac{s-\epsilon-\delta\nu+1}{2}\right)_j}
{\Gamma\left(\frac{s+\epsilon+\delta\nu}{2}\right)\Gamma\left(\frac{1-s+\epsilon+\delta\nu}{2}\right)}(\alpha\lambda^{-1})^{\frac12-s}ds\\
=(-1)^j\frac{\Gamma\left(\delta\nu\right)\sqrt{\pi}}{j! \left(1-\delta\nu\right)_j}\sum_{\substack{p\in\mathbb{Z}\pm\nu\\ \Re p\in\left[-\sigma,\sigma+1\right]}}\Res_{s=1-p}\left(\Lambda_f(s)\frac{\left(\frac{s+\epsilon-\delta\nu}{2}\right)_j \left(\frac{s-\epsilon-\delta\nu+1}{2}\right)_j}
{\Gamma\left(\frac{s+\epsilon+\delta\nu}{2}\right)\Gamma\left(\frac{1-s+\epsilon+\delta\nu}{2}\right)}(\alpha\lambda^{-1})^{\frac12-s}\right)\\
=(-1)^j\frac{\Gamma\left(\delta\nu\right)\sqrt{\pi}}{j! \left(1-\delta\nu\right)_j}\sum_{\substack{p\in\mathbb{Z}\pm\nu\\ \Re p\in\left[-\sigma,\sigma+1\right]}}\frac{\left(\frac{1-p+\epsilon-\delta\nu}{2}\right)_j \left(\frac{2-p-\epsilon-\delta\nu}{2}\right)_j}
{\Gamma\left(\frac{1-p+\epsilon+\delta\nu}{2}\right)\Gamma\left(\frac{p+\epsilon+\delta\nu}{2}\right)}\alpha^{p-\frac12}\lambda^{\frac12-p}\Res_{s=1-p}\Lambda_f(s).
\end{multline}
The final line follows because $\frac{\left(\frac{s+\epsilon-\delta\nu}{2}\right)_j \left(\frac{s-\epsilon-\delta\nu}{2}\right)_j}
{\Gamma\left(\frac{s+\epsilon+\delta\nu}{2}\right)\Gamma\left(\frac{1-s+\epsilon+\delta\nu}{2}\right)}$ and $\left(\alpha\lambda^{-1}\right)^{\frac12-s}$ are entire and the poles of $\Lambda_f(s)$ are simple. Since we assume $0<\sigma+|\Re\nu|<1$, the only values of $p$ that can occur in the above sum are in $\{\pm\nu,1\pm\nu\}$. The factor $\frac{\left(\frac{1-p+\epsilon-\delta\nu}{2}\right)_j \left(\frac{2-p-\epsilon-\delta\nu}{2}\right)_j}{\Gamma\left(\frac{1-p+\epsilon+\delta\nu}{2}\right)\Gamma\left(\frac{p+\epsilon+\delta\nu}{2}\right)}$ vanishes at some of these values: if $\epsilon=0$, then it vanishes at $-\delta\nu,1+\delta\nu$ and $1-\delta\nu$. If $\epsilon=1$, it vanishes at $1-\delta\nu$. We conclude that
\begin{multline}\label{eq:SumRes}
\mathcal{I}^{-\delta}_j\left(\alpha\lambda^{-1}\right)=
(-1)^j\frac{\Gamma\left(\delta\nu\right)\sqrt{\pi}}{j! \left(1-\delta\nu\right)_j}\\
\cdot\begin{cases}
\frac{\left(\frac12-\delta\nu\right)_j\left(1-\delta\nu\right)_j}{\Gamma(\delta\nu)\sqrt{\pi}}\alpha^{\delta\nu-\frac12}\lambda^{\frac12-\delta\nu}\Res_{s=1-\delta\nu}\Lambda_f(s), & \epsilon=0,\\
\sum_{p\in\{\pm\delta\nu,1+\delta\nu\}}\frac{\left(\frac{2-p-\delta\nu}{2}\right)_j \left(\frac{1-p-\delta\nu}{2}\right)_j}
{\Gamma\left(\frac{2-p+\delta\nu}{2}\right)\Gamma\left(\frac{p+1+\delta\nu}{2}\right)}\alpha^{p-\frac12}\lambda^{\frac12-p}\Res_{s=1-p}\Lambda_f(s), & \epsilon=1.
\end{cases}
\end{multline}
Assume that $\epsilon=0$. Equation \eqref{eqn:residue at s} becomes
\begin{align}
\begin{split}
\Res_{s=s_0+\delta\nu}&
\left(\sum_{\lambda\in T_{\beta,2\ell_0}}c_\lambda\lambda^{2s-2t_0-1}
\Lambda_f\left(s-t_0,\alpha\lambda^{-1},\cos\right)\right)\\
&=\Res_{s=s_0+\delta\nu}\left(i^{-[t_0]}(N\alpha^2)^{s-\frac12}\alpha^{-t_0}
\frac{(2\pi i)^{t_0}}{t_0!}
\Lambda_g\left(s,\beta,\cos^{([t_0])}\right)\frac
{
\gamma_f^{+}(1-s+t_0)
}
{
\gamma_f^{+}(1-s+[t_0])
}\right)\\
&\quad+
(-1)^j\frac{\left(\frac12-\delta\nu\right)_j}{j!}\alpha^{s_0-t_0-2\delta\nu-1}\Res_{s=1-\delta\nu}\Lambda_f(s)\left(\sum_{\lambda\in T_{\beta,2\ell_0}}c_\lambda\lambda^{s_0-t_0}\right).
\end{split}
\end{align}
Equation \ref{eq:resnot0} now follows from equation \eqref{Vandermonde}. Indeed, by assumption, $s_0-t_0\in\mathbb{Z}$ satisfies $0>s_0-t_0>-2\ell_0$. We also note that the quotient $\gamma_f^+(1-s+t_0)/\gamma_f^+(1-s+[t_0])$ does not have a pole at $s_0+\delta\nu$, so it does not contribute to the residue.

Now assume that $\epsilon=1$. In this case, we have

\begin{align}\label{eq:Iepsilon1}
\begin{split}
\mathcal{I}^{-\delta}_j\left(\alpha\lambda^{-1}\right)=(-1)^j\left[\frac{\left(\frac12-\delta\nu\right)_j\Gamma(\delta\nu)\sqrt{\pi}}{j!\Gamma\left(\delta\nu+\frac12\right)}\alpha^{\delta\nu-\frac12}\lambda^{\frac12-\delta\nu}\Res_{s=1-\delta\nu}\Lambda_f(s)\right.\\
+\frac{\left(\frac12\right)_j}{\delta\nu(1-\delta\nu)_j}\alpha^{-\delta\nu-\frac12}\lambda^{\frac12+\delta\nu}\Res_{s=1+\delta\nu}\Lambda_f(s)\\
\left.-\frac{\left(\frac12-\delta\nu\right)_j}{j!(-\delta\nu+j)}\alpha^{\delta\nu+\frac12}\lambda^{-\delta\nu-\frac12}\Res_{s=-\delta\nu}\Lambda_f(s)\right].
\end{split}
\end{align}
It follows from equation \eqref{eqn:residue at s} that:
\begin{align}
\begin{split}
i\pi\Res_{s=s_0+\delta\nu}&
\left(\sum_{\lambda\in T_{\beta,2\ell_0}}c_\lambda\lambda^{2s-2t_0-1}
\Lambda_f\left(s-t_0,\alpha\lambda^{-1},\sin\right)\right)\\
&=\Res_{s=s_0+\delta\nu}\left(i^{-[1+t_0]}(N\alpha^2)^{s-\frac12}\alpha^{-t_0}
\frac{(2\pi i)^{t_0}}{t_0!}
\Lambda_g\left(s,\beta,\cos^{([1+t_0])}\right)\frac
{
\gamma_f^{-}(1-s+t_0)
}
{
\gamma_f^{-}(1-s+[t_0])
}\right)\\
&\quad+(-1)^j
\sum_{\lambda\in T_{\beta,2\ell_0}}c_\lambda \lambda^{s_0-t_0}\frac{\left(\frac{1-2\delta\nu}{2}\right)_j\Gamma(\delta\nu)\sqrt{\pi}}{j!\Gamma\left(\delta\nu+\frac12\right)}\alpha^{s_0-t_0+2\delta\nu-1}\Res_{s=1-\delta\nu}\Lambda_f(s)\\
&\quad+(-1)^j \sum_{\lambda\in T_{\beta,2\ell_0}}c_\lambda \lambda^{s_0-t_0+2\delta\nu}\frac{\left(\frac12\right)_j}{\delta\nu(1-\delta\nu)_j}\alpha^{s_0-t_0-1}\Res_{s=1+\delta\nu}\Lambda_f(s)\\
&\quad-(-1)^j \sum_{\lambda\in T_{\beta,2\ell_0}}c_\lambda\lambda^{s_0-t_0-1}\frac{\left(\frac{1-2\delta\nu}{2}\right)_j}{j!(-\delta\nu+j)}\alpha^{s_0-t_0+2\delta\nu}\Res_{s=-\delta\nu}\Lambda_f(s).\\
\end{split}
\end{align}
The stated equation \eqref{eq:nunot0epsilon1} now follows from \eqref{Vandermonde}. Note that the last term vanishes since $-t_0>s_0-t_0-1 > -2\ell_0$ .

\end{proof}

\begin{proposition}\label{AnalyticProperties.proposition}
Make the assumptions of Theorem~\ref{thm:PolarConverse}. If $\epsilon=1$, then $\Lambda_g(s)$ and $\Lambda_f(s)$ continue to entire functions on $\mathbb{C}$.
\end{proposition}
\begin{proof}
Since we already established that the only poles $\Lambda_f$ and $\Lambda_g$ can have are simple and in the set $\{\pm\nu,1\pm\nu\}$ it suffices to show that the residues of these functions vanish there.

Let $\beta=-\frac{1}{Nq}$, $T_{\beta, M}$ a set of cardinality $M$ satisfying \eqref{Vandermonde}, and $\lambda=p\in T_{\beta,2\ell_0}$, so that $\alpha=q$ and $\alpha\lambda^{-1}=\frac{q}{p}$. By absolute convergence of the Dirichlet series we know that $\Lambda\left(s-t_0,\alpha\lambda^{-1},\sin\right)$ is holomorphic for $\Re(s-t_0)\geq \sigma+1$, and so by the functional equation it is also holomorphic for $\Re(s-t_0)\leq-\sigma$. When $s_0=0$, for all even $t_0\geq\sigma+1$ equation \eqref{eq:nunot0epsilon1} simplifies to:
\begin{multline}\label{eq:reside}
iN^{\delta\nu-\frac12}\Res_{s=\delta\nu}\Lambda_g\left(s,\beta,\sin\right)+\frac{\Gamma(\delta\nu)\sqrt{\pi}}{\Gamma\left(\delta\nu+\frac12\right)}\Res_{s=1-\delta\nu}\Lambda_f(s)\\
=\frac{t_0!}{(1-2\delta\nu)_{t_0}}\frac{\alpha^{-2\delta\nu}}{\delta\nu}\left(\sum_{\lambda\in T_{\beta,M}}c_{\lambda}\lambda^{-t_0+2\delta\nu}\right)\Res_{s=1+\delta\nu}\Lambda_f(s).
\end{multline}
We used the formulas $\gamma_f^+(1+t_0-\delta\nu)/\gamma_f^+(1+[t_0]-\delta\nu) = \pi^{-t_0}(\frac12-\delta\nu)_{t_0/2}(\frac12)_{t_0/2}$ and $2^{t_0}(\frac12)_{t_0/2}(1)_{t_0/2}=t_0!$ above.
The left-hand side of equation \eqref{eq:reside} does not depend on $t_0$. Recall that $T_{\beta,M}$ and the $c_\lambda$ were chosen such that \eqref{Vandermonde} is satisfied. We want to show that we can add an element $\lambda_0\in T_\beta$ to $T_{\beta,M}$, so that \eqref{Vandermonde} is still satisfied and in addition $\sum_{\lambda\in T_{\beta,2\ell_0}} c_\lambda \lambda^{-t_0+2\delta\nu}$ assumes an arbitrary value.

Equivalently, we want to find $\lambda_0\in T_\beta$ such that the vectors $(\lambda^{-t})_{\lambda\in T_{\beta,M}\cup\{\lambda_0\}}$ for $t\in\{0,1,\ldots,M-1\}$ are linearly independent of the vector $(\lambda^{-t_0+2\delta\nu})_{\lambda\in T_{\beta,M}\cup\{\lambda_0\}}$. Consider the matrix that has these $M+1$ vectors in $\R^{M+1}$ as Columns. Developing the determinant with respect to the last row we obtain an expression in $\lambda_0$ of the form
\begin{equation}\label{eq:extended Vandermonde determinant}
\lambda_0^{-t_0+2\delta\nu}c + P(\lambda_0),
\end{equation}
where $c$ is a non-zero constant (the Vandermonde determinant of $T_{\beta,M}$) and $P$ is a polynomial with complex coefficients of degree $M-1$. Since $T_\beta$ is an infinite set we can choose $\lambda_0$ arbitrarly large. Suppose the expression \eqref{eq:extended Vandermonde determinant} vanishes for all $\lambda_0\in T_\beta$. By comparing the growth of the two terms in \eqref{eq:extended Vandermonde determinant} for $\lambda_0\to\infty$, we conclude that $P = d\lambda_0^{-t_0}$ and $\nu$ is purely imaginary. Now comparing the argument of the two terms we arrive at a contradiction. Hence there exists a $\lambda_0\in T_\beta$ such that \eqref{eq:extended Vandermonde determinant} is non-zero. So we can apply Lemma \ref{Residues.lemma} with $T_{\beta,M}\cup \{\lambda_0\}$ instead of $T_{\beta,M}$ and choose coefficients $c_\lambda$ for $\lambda\in T_{\beta,M}\cup\{\lambda_0\}$ such that \eqref{Vandermonde} and \eqref{eq:reside} is satisfied and $\sum_{\lambda\in T_{\beta,M}\cup\{\lambda_0\}} c_\lambda \lambda^{-t_0+2\delta\nu}=1$. We can also choose coefficients $c'_\lambda$ with $\sum_{\lambda\in T_{\beta,M}\cup\{\lambda_0\}} c_\lambda \lambda^{-t_0+2\delta\nu}=2$ such that \eqref{eq:reside} is satisfied with $c_\lambda$ replaced by $c'_\lambda$. We conclude $\Res_{s=1+\delta\nu}\Lambda_f(s)=0$.

As $\delta\in\{\pm\}$ was arbitrary, we have shown that $\Res_{s=1+\nu}\Lambda_f(s)=\Res_{s=1-\nu}\Lambda_f(s)=0$. Reversing the roles of $f$ and $g,$ we deduce the same for $\Lambda_g(s)$. By the functional equation, we have $\Res_{s=\pm\nu}\Lambda_f(s)=\Res_{s=\pm\nu}\Lambda_g(s)=0$.

\end{proof}
Note that equation \eqref{eq:reside} also implies that 
\begin{equation}
\Res_{s=\delta\nu}\Lambda_g\left(s,-\frac{1}{Nq},\sin\right)=-iN^{\frac{1}{2}-\delta\nu}\frac{\delta\nu\sqrt{\pi}}{\Gamma\left(\delta\nu+\frac12\right)}\Res_{s=1-\delta\nu}\Lambda_f(s),
\end{equation}
and so the residue on the left-hand side is independent of $q$.
\begin{proposition}
Make the assumptions of Theorem~\ref{thm:PolarConverse}. If $\psi$ is a primitive Dirichlet character with conductor $q\in\mathcal{P}$, then $\Lambda_f(s,\psi)$ and $\Lambda_g(s,\psi)$ continue to elements of $\mathcal{H}(-\infty,\infty)$.
\end{proposition}
\begin{proof}
For $\psi$ as in the statement, recall that
\begin{equation}\label{eq:CharacterCosExpansion}
\psi(n)=(-i)^{\sgn\psi}\frac{\tau(\psi)}{q}\sum_{b\text{ mod }q}\bar{\psi}(-b)\cos^{(\sgn\psi)}\left(2\pi\frac{bn}{q}\right),
\end{equation}
and so
\begin{equation}\label{eq:CharacterSumAdditive}
\Lambda_g(s,\psi)
=(-i)^{\sgn\psi}\frac{\tau(\psi)}{q}
\sum_{b\text{ mod }q}
\bar{\psi}(-b)\Lambda_g\left(s,\frac{b}{q},\cos^{(\sgn\psi)}\right).
\end{equation}
By our assumption on $\sigma$, we know that $\Lambda_f(s,\psi)$ are entire for $\Re s\geq \frac32>1+\sigma$. By equation~\eqref{eq:FE} it suffices to prove that 
\begin{equation}\label{eq:want}
\Res_{s=\delta\nu}\Lambda_f\left(s,\psi\right)=\Res_{s=\delta\nu}\Lambda_g\left(s,\psi\right)=0, 
\end{equation}
for $\delta\in\{\pm\}$. 

Let $\alpha>0$ and $T_{\beta,M}$ be as in Lemma \ref{lem:isolated term}. Let $s_0<1$ and choose $t_0>1$ such that $t_0-s_0$ is odd. Taking the residue at $s=s_0+\delta\nu$ of equation \eqref{eq:isolatedterm} we obtain
\begin{multline}\label{eq:1}
i^{-[\epsilon+t_0]}(N\alpha^2)^{s_0+\delta\nu-\frac12}\alpha^{-t_0}
\frac{(2\pi i)^{t_0}}{t_0!}
\frac
{
\gamma_f^{(-)^\epsilon}(1-s_0+t_0-\delta\nu)
}
{
\gamma_f^{(-)^\epsilon}(1-s_0+[t_0]-\delta\nu)
}\Res_{s=s_0+\delta\nu}\Lambda_g\left(s,\beta,\cos^{[\epsilon+t_0]}\right)\\
=(-i\pi)^\epsilon\sum_{\lambda\in T_{\beta,2\ell_0}}c_\lambda\lambda^{2s_0-2t_0+2\delta\nu-1}
\Res_{s=s_0+\delta\nu}\Lambda_f\left(s-t_0,\alpha\lambda^{-1},\cos^{(\epsilon)}\right).
\end{multline}
First assume $\epsilon=1$. If $\beta=\frac{b}{Nq}$ for $b<0$ coprime to $Nq$, then $\alpha\lambda^{-1}=\frac{q}{p}$ for some $p\equiv b$ mod $Nq$. By the functional equation for $\Lambda_f\left(s-t_0,\alpha\lambda^{-1},\sin\right)$, the second line of \eqref{eq:1} is zero and so $\Res_{s=s_0+\delta\nu}\Lambda_g\left(s,\frac{b}{Nq},\cos^{(s_0)}\right)=0$ for all $s_0<1$. If $\beta=\frac{b}{Nq}$ with $b>0$ coprime to $Nq$, then we choose $b'\in\mathcal{P}$ with $b'\equiv -b\bmod Nq$, so that $\Lambda_g\left(s,\frac{b}{Nq},\cos^{(s_0)}\right) = \Lambda_g\left(s,\frac{-b'}{Nq},\cos^{(s_0)}\right)$ and so $\Res_{s=s_0+\delta\nu}\Lambda_g\left(s,\beta,\cos^{(s_0)}\right)=0$ for all $\beta$ of the form $b/Nq$.

To obtain information on the other additive twists we use Lemma \ref{Residues.lemma}. Let $s_0<1$, $t_0>1$ such that $[s_0]=[t_0]$. Equation \eqref{eq:nunot0epsilon1} and Proposition \ref{AnalyticProperties.proposition} imply:
\begin{multline}\label{eq:2}
i^{-[1+t_0]}(N\alpha^2)^{s_0-\frac12+\delta\nu}\alpha^{-t_0}
\frac{(2\pi i)^{t_0}}{t_0!}\frac
{
\gamma_f^{-}(1-s_0+t_0-\delta\nu)
}
{
\gamma_f^{-}(1-s_0+[t_0]-\delta\nu)
}\Res_{s=s_0+\delta\nu}\Lambda_g\left(s,\beta,\cos^{([1+t_0])}\right)\\
=i\pi\sum_{\lambda\in T_{\beta,M}}c_\lambda\lambda^{2s_0+2\delta\nu-2t_0-1}
\Res_{s=s_0+\delta\nu}\Lambda_f\left(s-t_0,\alpha\lambda^{-1},\sin\right)\\
\end{multline}
As above, for $\beta = \frac{b}{Nq}$ the second line vanishes and we conclude $\Res_{s=s_0+\delta\nu}\Lambda_g\left(s,\frac{b}{Nq},\cos^{(1+s_0)}\right)=0$ for all $s_0<1$.

If $\beta=\frac{b}{q}$ for $b<0$ coprime to $q$, then $\lambda^{-1}\alpha=\frac{q}{Np}$ for a prime $p$ congruent to $-b$ modulo $q$. By the previous paragraph the second lines of equations \eqref{eq:1} and \eqref{eq:2} both vanish. Therefore, equation \eqref{eq:1} for $s_0=0$ and $t_0=3$ implies that $\Res_{s=\delta\nu}\Lambda_g\left(s,\frac{b}{q},\cos\right)=0$ and equation \eqref{eq:2} for $s_0=0$ and $t_0=2$ implies that $\Res_{s=\delta\nu}\Lambda_g\left(s,\frac{b}{q},\sin\right)=0$. Reversing the roles of $f$ and $g$, we deduce the same for $\Res_{s=\delta\nu}\Lambda_f\left(s,\frac{b}{q},\cos\right)$ and $\Res_{s=\delta\nu}\Lambda_f\left(s,\frac{b}{q},\sin\right)$. By equation \eqref{eq:CharacterSumAdditive}, we deduce equation \eqref{eq:want} as required.

Now consider $\epsilon=0$. Let $t_0>1$ and $s_0<1$ be integers. Let $t_0-s_0$ be odd, and $\beta,\beta'\in\Q_{<0}$ with the same numerator. In \eqref{eq:1} we choose the set $T=T_{\beta,2\ell_0}=T_{\beta',2\ell_0}$ to be a subset of $T_\beta\cap T_{\beta'}$. This is possible since $T_\beta\cap T_{\beta'}$ is infinite. Hence \eqref{eq:1} applies to the pair $\beta$ and $\alpha=-1/N\beta$ and the pair $\beta'$ and $\alpha'=-1/N\beta$. Subtracting the resulting equations from each other we obtain
\begin{multline}\label{eq:eps0 t0-s0 odd}
i^{[t_0]}N^{s_0+\delta\nu-\frac12}
\frac{(2\pi i)^{t_0}}{t_0!}
\frac
{
\gamma_f^{+}(1-s_0+t_0-\delta\nu)
}
{
\gamma_f^{+}(1-s_0+[t_0]-\delta\nu)
}\\
\cdot\left[\alpha^{2s_0+2\delta\nu-t_0-1}\Res_{s=s_0+\delta\nu}\Lambda_g\left(s,\beta,\cos^{[t_0]}\right)-\alpha'^{2s_0+2\delta\nu-t_0-1}\Res_{s=s_0+\delta\nu}\Lambda_g\left(s,\beta',\cos^{[t_0]}\right)\right]\\
=\sum_{\lambda\in T}c_\lambda\lambda^{2s_0+2\delta\nu-2t_0-1}
\Res_{s=s_0+\delta\nu}\left[\Lambda_f\left(s-t_0,\alpha\lambda^{-1},\cos\right)-\Lambda_f\left(s-t_0,\alpha'\lambda^{-1},\cos\right)\right].
\end{multline}
When $\beta = \frac{b}{Nq}$ and $\beta' = \frac{b}{Nq'}$ with $b<0$ coprime to $Nq$ the last line vanishes, since $\Lambda_f\left(s-t_0,\alpha\lambda^{-1},\cos\right)-\Lambda_f\left(s-t_0,\alpha'\lambda^{-1},\cos\right)$ is a linear combination of twists of $\Lambda_f$ by characters and hence holomorphic at $s_0-t_0+\delta\nu$, since $s_0-t_0<-2$. Therefore, we have 
\[
\alpha^{s_0+2\delta\nu-t_0-1}\Res_{s=s_0+\delta\nu}\Lambda_g\left(s,\frac{b}{Nq},\cos^{[t_0]}\right)=\alpha'^{s_0+2\delta\nu-t_0-1}\Res_{s=s_0+\delta\nu}\Lambda_g\left(s,\frac{b}{Nq'},\cos^{[t_0]}\right).
\]
Varying $t_0$ we deduce that $\Res_{s=s_0+\delta\nu}\Lambda_g\left(s,\beta,\cos^{[t_0]}\right)=\Res_{s=s_0+\delta\nu}\Lambda_g\left(s,\beta,\cos^{(s_0+1)}\right)=0$ for all $s_0<1$ and all $\beta$ of the form $\frac{b}{Nq}$ with $b$ coprime to $Nq$. We can omit the condition $b<0$ by the same argument as in the case $\epsilon=1$.

On the other hand, if $t_0-s_0$ is even we consider the equations \eqref{eq:resnot0} for $\beta$ and $\beta'$. Subtracting the two equations from each other we note that when $\beta = \frac{b}{Nq}$ and $\beta' = \frac{b}{Nq'}$, again the terms $\Res_{s_0+\delta\nu}\left(\Lambda_f\left(s-t_0,\alpha\lambda^{-1},\cos\right)-\Lambda_f\left(s-t_0,\alpha'\lambda^{-1},\cos\right)\right)$ vanish and so we are left with
\begin{align}\label{eq:entirety 1}
\begin{split}
&i^{-[t_0]}N^{s_0-\frac12+\delta\nu}
\frac{(2\pi i)^{t_0}}{t_0!}\frac
{
\gamma_f^{+}(1-s_0+t_0-\delta\nu)
}
{
\gamma_f^{+}(1-s_0+[t_0]-\delta\nu)
}
\\
&\quad\cdot\left(
\alpha^{2s_0+2\delta\nu-1-t_0}\Res_{s=s_0+\delta\nu}\Lambda_g\left(s,\beta,\cos^{([t_0])}\right)-\alpha'^{2s_0+2\delta\nu-1-t_0}\Res_{s=s_0+\delta\nu}\Lambda_g\left(s,\beta',\cos^{([t_0])}\right)\right)\\
&\quad\quad=(-1)^{j+1}\delta_{0}(s_0)\frac{\left(\frac12-\delta\nu\right)_j}{j!}(\alpha^{s_0-t_0-1+2\delta\nu}-\alpha'^{s_0-t_0-1+2\delta\nu})\Res_{s=1-\delta\nu}\Lambda_f(s)
\end{split}
\end{align}
where $j=\frac{1}{2}(t_0-s_0)$. We see that  
\[q^{2s_0+2\delta\nu-1-t_0}\Res_{s=s_0+\delta\nu}\Lambda_g(s,\frac{b}{Nq},\cos^{(s_0)}) = q'^{2s_0+2\delta\nu-1-t_0}\Res_{s=s_0+\delta\nu}\Lambda_g(s,\frac{b}{Nq'},\cos^{(s_0)})
\]
for $s_0<0$, since in that case the last line of \eqref{eq:entirety 1} vanishes. Varying $t_0$ we see that $\Res_{s=s_0+\delta\nu}\Lambda_g(s,\frac{b}{Nq},\cos^{(s_0)})=0$ for $s_0<0$.

Now let $\beta = \frac{b}{q}$ with $b<0$, which implies $\lambda^{-1}\alpha=\frac{q}{Np}$ for $\lambda\in T_\beta$. We first insert $s_0=0$ and $t_0=3$ into \eqref{eq:1}. Since we have shown above that $\Res_{s=\delta\nu}\Lambda_f(s-3,\lambda^{-1}\alpha,\cos)=\Res_{s=-3+\delta\nu}\Lambda_f(s,\lambda^{-1}\alpha,\cos)=0$ we see $\Res_{s=\delta\nu}\Lambda_g(s,\frac{b}{q},\sin)=0$. By equation \eqref{eq:CharacterSumAdditive} this implies that $\Res_{s=\delta\nu}\Lambda_g(s,\psi)=0$ for all odd characters $\psi$ of conductor $q$. For the even twists consider \eqref{eq:resnot0} for $s_0=0$ and $t_0=2$. By our previous considerations the first line vanishes and we are left with
\[
N^{-\frac12+\delta\nu}
\frac{(2\pi i)^{2}}{2}\frac
{
\gamma_f^{+}(3-\delta\nu)
}
{
\gamma_f^{+}(1-\delta\nu)
}\Res_{s=\delta\nu}
\Lambda_g\left(s,\beta,\cos\right)
=\left(\frac12-\delta\nu\right)\Res_{s=1-\delta\nu}\Lambda_f(s).
\]
Hence $\Res_{s=\delta\nu}\Lambda_g(s,\frac{b}{q},\cos)$ is independent of $b$ coprime to $q$. Again we deduce that $\Res_{s=\delta\nu}\Lambda_g(s,\psi)=0$ from equation \eqref{eq:CharacterSumAdditive} . Reversing the roles of $f$ and $g$ and using the functional equation we finally conclude \eqref{eq:want}.

\end{proof}

\section{Proof of Corollary~\ref{thm:Quotients}}\label{sec.Quotients}

Let $\psi$ be a primitive Dirichlet character of prime conductor $q$. Define $\epsilon_{\psi}\in\{0,1\}$ by $\chi(-1)=(-1)^{\epsilon_{\psi}}$. The automorphic representation $\omega_{\psi}$ of GL$_1(\mathbb{A}_{\mathbb{Q}})$ associated to $\psi$ 
defines an $L$-function $\Lambda(\omega_{\psi},s)$ with functional equation
\begin{equation}\label{eq.FEDirichlet}
\Lambda\left(\omega_{\psi},s\right)=(-i)^{\epsilon_{\psi}}\frac{\tau\left(\psi^{-1}\right)}{q^{\frac12}}q^{-s}\Lambda\left(\overline{\omega_{\psi}},1-s\right).
\end{equation} 
Let $f$ be as in Corollary~\ref{thm:Quotients}, and let $\pi_f$ denote the automorphic representation of GL$_2(\mathbb{A}_{\mathbb{Q}})$ corresponding to $f$. The symmetric square Sym$^2f$ defines a, by assumption cuspidal, automorphic representation of GL$_3(\mathbb{A}_{\mathbb{Q}})$ and we denote its $L$-function by $L(\Sym^2f,s)=\sum_{n=1}^{\infty}c_nn^{-s}$. By the Jacquet--Shalika bound  (cf. \cite{JacquetShalika1981} and \cite[Appendix]{RS}), we have $|c_n|=O(n^{1/2+\kappa})$, for all $\kappa>0$. The completed symmetric square $L$-function is given by
\[\Lambda\left(\Sym^2f,s\right) = 
\pi^{-3s/2}\Gamma\left(\frac{s+\nu}{2}\right)\Gamma\left(\frac{s-\nu}{2}\right)\Gamma\left(\frac{s}{2}\right)L(\Sym^2f,s)
=\frac{\Lambda\left(\pi_f\times\pi_f,s\right)}{\xi(s)},\]
where $\xi(s)=\Gamma_{\mathbb{R}}(s)\zeta(s)$ is the completed Riemann zeta function. The contragredient of Sym$^2f$ is $\overline{\Sym^2f}$ and  we have the functional equation
\begin{equation}\label{eq.FEsym2}
\Lambda(\Sym^2f,s)=M^{-s}\Lambda\left(\overline{\Sym^2f},1-s\right),
\end{equation}
for some integer $M$, the conductor of $\Sym^2f$, which divides $N^2$ \cite[Equation~(5.100)]{IK}. 
Let $q$ be a prime integer not dividing $M$ and let $\psi$ be a primitive Dirichlet character of conductor $q$. Because we have assumed $f$ to be self-dual, \cite[page~473]{RBAR23} implies that the twisted automorphic representation Sym$^2f\times\omega_\psi$ of GL$_3(\mathbb{A}_{\mathbb{Q}})$ has $L$-function
\begin{equation}
\Lambda\left(\Sym^2f\times\omega_\psi,s\right)=\frac{\Lambda\left((\pi_f\otimes\omega_{\psi})\times\pi_f,s\right)}{\Lambda(\omega_\psi,s)}.
\end{equation}
For $\Re(s)\gg0$, we have
\begin{equation}\label{eq.sym2twistcomplete}
\Lambda\left(\Sym^2f\times\omega_\psi,s\right)=\pi^{-3s/2}\Gamma\left(\frac{s+\epsilon_{\psi}+\nu}{2}\right)\Gamma\left(\frac{s+\epsilon_{\psi}-\nu}{2}\right)\Gamma\left(\frac{s+\epsilon_{\psi}}{2}\right)\sum_{n=1}^{\infty}\overline{\psi(n)}c_nn^{-s}.
\end{equation}
Note that the Dirichlet coefficients of this $L$-function are the Dirichlet coefficients of $L(\Sym^2 f,s)$ twisted by $\overline{\psi}$.
The contragredient of Sym$^2f\times\omega_\psi$ is given by $\overline{\Sym^2f}\times\overline{\omega_\psi}$ and we have functional equation
\begin{equation}\label{eq.FEsym2twist}
\Lambda\left(\Sym^2f\times\omega_\psi,s\right)=i^{3\epsilon_{\psi}}\psi^{-1}(M)\frac{\tau(\psi^{-1})^3}{q^{3/2}}\left(Mq^3\right)^{-s}\Lambda\left(\overline{\Sym^2f}\times\overline{\omega_\psi},1-s\right).
\end{equation}
\begin{lemma}\label{lem.DCofDC}
Let $c_n,d_n$ be two sequences such that $|c_n|,|d_n|=O(n^\sigma)$ and let $a_n$ be the Dirichlet convolution of $c_n$ and $d_n$, i.e. $a_n = \sum_{d|n} c_d d_{n/d}$. Then $|a_n| = O(n^{\sigma+\kappa})$ for all $\kappa>0$.
\end{lemma}
\begin{proof} If $c_n\leq Cn^\sigma$ and $d_n\leq C' n^\sigma$. Then
\[
|a_n| \leq \sum_{d|n} Cd^\sigma\cdot C' (n/d)^\sigma = CC' \tau(n)n^\sigma = O(n^{\sigma+\kappa}).
\]
\end{proof}
It follows from Lemma~\ref{lem.DCofDC} that the quotient $L(\Sym^2f,s)/\zeta(s) = \sum_{n=1}^{\infty} a_n n^{-s}$ satisfies $a_n = O(n^{\frac12+\kappa})$, for all $\kappa>0$. Similarly, we define $L(\overline{\Sym^2f},s)/\zeta(s)=\sum_{n=1}^{\infty}b_nn^{-s}$ with $|b_n|=O(n^{\frac12+\kappa})$, for all $\kappa>0$. For any primitive Dirichlet character $\psi$ of conductor $q$ not dividing $M$, we have
\begin{align}\label{eq.QuotientGammaDirichlet}
\begin{split}
\frac{\Lambda\left(\Sym^2f\times\omega_\psi,s\right)}{\Lambda\left(\omega_{\psi},s\right)}=\Gamma_{\mathbb{R}}\left(s+\epsilon_{\psi}+\nu\right)\Gamma_{\mathbb{R}}\left(s+\epsilon_{\psi}-\nu\right)\sum_{n=1}^{\infty}\overline{\psi(n)}a_nn^{-s},\\
\frac{\Lambda\left(\overline{\Sym^2f}\times\overline{\omega_\psi},s\right)}{\Lambda\left(\overline{\omega_{\psi}},s\right)}=\Gamma_{\mathbb{R}}\left(s+\epsilon_{\psi}+\nu\right)\Gamma_{\mathbb{R}}\left(s+\epsilon_{\psi}-\nu\right)\sum_{n=1}^{\infty}\psi(n)b_nn^{-s}.
\end{split}
\end{align}
Let $\mathcal{P}$ denote the set of primes not dividing $M$. For any primitive Dirichlet character of conductor $q\in\mathcal{P}\cup\{1\}$, dividing equation~\eqref{eq.FEsym2twist} by equation~\eqref{eq.FEDirichlet} gives
\begin{equation}\label{eq.FEquotient}
\frac{\Lambda\left(\Sym^2f\times\omega_\psi,s\right)}{\Lambda\left(\omega_{\psi},s\right)}=(-1)^{\epsilon_{\psi}}\frac{\tau(\psi^{-1})^2}{q}\left(Mq^2\right)^{-s}\frac{\Lambda\left(\overline{\Sym^2f}\times\overline{\omega_\psi},1-s\right)}{\Lambda\left(\overline{\omega_{\psi}},1-s\right)}.
\end{equation}
Keeping in mind equation~\eqref{eq.QuotientGammaDirichlet}, we see that equation~\eqref{eq.FEquotient} reduces to equation~\eqref{eq:FE} with $\epsilon=0$.

 The quotients in equation~\eqref{eq.QuotientGammaDirichlet} extend to meromorphic functions on $\mathbb{C}$. If $\Lambda\left(\Sym^2f,s\right)/\xi(s)$ has only finitely many poles, then Theorem~\ref{thm:PolarConverse} implies that
\[\frac{\Lambda\left(\Sym^2f,s\right)}{\xi(s)}=\Lambda_h(s)\]
for some even Maass form $h$ of level $M$, trivial nebentypus and eigenvalue equal to that of $f$. After cancelling the gamma functions, we are left with an identity of Dirichlet series
\[L\left(\Sym^2f,s\right)=L_h(s)\zeta(s).\]
The Maass form $h$ can be written as a finite linear combination of Hecke eigenforms\footnote{In fact, one could show that $h$ itself is a Hecke eigenform by generalising \cite[Theorem~1.1]{BT} to Maass forms. Such a generalisation is discussed in [loc. cit., Remark 3]. Upon establishing that, Corollary~\ref{thm:Quotients} would follow from \cite[Theorem~1.1]{factorisationLfunctions}.} $h_i$, that is $h(z)=\sum_{i=1}^m\alpha_i h_i(z)$. Therefore, $L_h(s)=\sum_{i=1}^m\alpha_iL_{h_i}(s).$  Each $L_{h_i}(s)$ has an Euler product factorisation. This contradicts the linear independence of automorphic $L$-functions as proved in \cite{KMP}.
\appendix

\section{~}
\subsection{The Gauss hypergeometric function}\label{Hypergeometric.subsection}
The Mellin transforms of twisted Maass forms yielded the hypergeometric function $_2F_1$, which is defined initially on $|z|<1$ by the following power series
\begin{equation}
\hyp{a}{b}{c}{z} = \sum_{n=0}^{\infty}\frac{(a)_n(b)_n}{(c)_nn!}z^n,
\end{equation}
where $(x)_n$ is the rising Pochhammer symbol, that is,
\[
(x)_n=\frac{\Gamma(x+n)}{\Gamma(x)}=x(x+1)\cdots(x+n-1).
\] 
The function $\hyp{a}{b}{c}{z}$ satisfies the so-called Euler identity:
\begin{align}\label{eqn:euler identity}
\hyp{a}{b}{c}{z} = (1-z)^{c-a-b} \hyp{c-a}{c-b}{c}{z}.
\end{align}
We also have \cite[(A.11)]{Mezhericher}:
\begin{align}\label{eqn:hypergeo transform}
\hyp{a}{b}{c}{z} = (1-z)^{-a}\hyp{a}{c-b}{c}{\frac{z}{z-1}},\quad\text{for }z\notin(1,\infty).
\end{align}
Hence for $w\in\R$
\begin{equation}\label{eqn:hypergeo transform 2}
\hyp{\frac{\epsilon+s+\nu}{2}}{\frac{\epsilon+s-\nu}{2}}{\frac{1+2\epsilon}{2}}{-w^2} = 
(1+w^2)^{-\frac{\epsilon+s+\nu}{2}} \hyp{\frac{\epsilon+s+\nu}{2}}{\frac{\epsilon+1-s+\nu}{2}}{\frac{1+2\epsilon}{2}}{\frac{w^2}{1+w^2}}.
\end{equation}
Define $z = \frac{1-w^2}{1+w^2}$. We introduce some notation from \cite[\S 7.1]{Luke1969}. Let $z=\cosh(i\nu)$, where $\nu\in[0,\pi)$. In this special case the sector $\mathcal{Q}$ from \cite[\S 7.1]{Luke1969} is just $|\arg\lambda|\leq \pi - \delta$ for some $\delta>0$. In particular $\lambda$ lies in it if $\Im \lambda \to \infty$. Equation (8) from \textit{loc}.\textit{cit}. tells us how the hypergeometric function above grows when $\Im \lambda\to\infty$: 
\begin{multline*}
\hyp{a+\lambda}{b-\lambda}{c}{\frac{1-z}{2}}\sim\frac{2^{a+b-1}\Gamma(1-b+\lambda)\Gamma(c)(1+e^{-i\nu})^{c-a-b-1/2}}{(\lambda\pi)^{1/2}\Gamma(c-b+\lambda)(1-e^{-i\nu})^{c-1/2}}\\
\times \left[e^{(\lambda-b)i\nu}+e^{\pm i \pi(c-1/2)-(\lambda+a)i\nu}+O(1/\lambda)\right]
\end{multline*}
For $s=\sigma+it$ we set $a=\frac{\epsilon+\sigma+\nu}{2}$, $b=\frac{\epsilon+1-\sigma+\nu}{2}$, $c=\frac{1+2\epsilon}{2}$ and $\lambda=\frac{it}{2}$. Recall Stirling's formula for fixed $\sigma$: 
\begin{equation}\label{eq:Stirling}
\Gamma(\sigma+it)=\sqrt{2\pi}(it)^{\sigma-1/2}e^{-\frac{\pi}{2}|t|}\left(\frac{|t|}{e}\right)^{it}\left(1+O(|t|^{-1}\right),~|t|\rightarrow\infty.
\end{equation}
We deduce that \eqref{eqn:hypergeo transform 2} grows at most like $e^{\frac{v}{2}t}$. Finally we conclude
\[
\Gamma\left(\frac{s+\nu+\epsilon}{2}\right)\Gamma\left(\frac{s-\nu+\epsilon}{2}\right)
\hyp{\frac{s+\nu+\epsilon}{2}}{\frac{s-\nu+\epsilon}{2}}{\frac12+\epsilon}{-w^2}
\]
decays exponentially as $|t|\rightarrow \infty$ for all $w\in\R$, since the exponential term in both Gamma factors is $e^{-\frac{\pi}{4}t}$ while the hypergeometric function has exponential term $e^{\frac{v}{2}t}$ and $v<\pi$. 

In section~\ref{PolarConverse.subsection} we used an explicit analytic continuation for the hypergeometric function outside its disc of convergence. When $|z|>1$ and $a-b\notin\mathbb{Z}$, one has that\footnote{\href{http://functions.wolfram.com/HypergeometricFunctions/Hypergeometric2F1/02/02/0001/}{wolfram.com/HypergeometricFunctions/Hypergeometric2F1/02/02/0001/}}:
\begin{multline}\label{eq:hypergeometric continuation}
\hyp{a}{b}{c}{z}=\frac{\Gamma(b-a)\Gamma(c)(-z)^{-a}}{\Gamma(b)\Gamma(c-a)}\sum_{k=0}^{\infty}\frac{(a)_k(a-c+1)_kz^{-k}}{k!(a-b+1)_k}\\
+\frac{\Gamma(a-b)\Gamma(c)(-z)^{-b}}{\Gamma(a)\Gamma(c-b)}\sum_{k=0}^{\infty}\frac{(b)_k(b-c+1)_kz^{-k}}{k!(b-a+1)_k}.
\end{multline}

\end{document}